\documentclass[10pt]{amsart}

\usepackage{amsmath}
\usepackage{array,multirow,graphicx}
\usepackage{amsfonts}
\usepackage{amssymb}
\usepackage{amsthm}
\usepackage{url}
\usepackage[all]{xy}
\usepackage{dsfont} 
\usepackage{nicematrix}
\usepackage{enumitem}
\usepackage{graphicx}
\usepackage{caption}
\usepackage{subcaption}
\usepackage{comment} 
\usepackage{stmaryrd}
\usepackage{hyperref}
\usepackage{todonotes}
\usepackage{color}
\usepackage{tikz-cd}
\usepackage[backend=bibtex,style=alphabetic,url=false,doi=false,isbn=false,sorting=nyt,maxbibnames=9,giveninits=true]{biblatex}
\usepackage{MnSymbol}
\usepackage{mathtools}
\allowdisplaybreaks
\usepackage{pgfplots}

 \AtEveryBibitem{\clearfield{month}}

\AtEveryBibitem{%
  \clearlist{language}%
}
\numberwithin{equation}{section}

\newtheorem{proposition}{Proposition}[section]
\newtheorem{lemma}[proposition]{Lemma}
\newtheorem{theorem}[proposition]{Theorem}
\newtheorem{corollary}[proposition]{Corollary}

\theoremstyle{definition}
\newtheorem{remark}[proposition]{Remark}
\newtheorem{definition}[proposition]{Definition}

\DeclareMathOperator{\Aut}{Aut}

\DeclareMathOperator{\DF}{DF}

\DeclareMathOperator{\GIT}{GIT}
\DeclareMathOperator{\Product}{Prod}

\newcommand{\N}{\mathbb{N}}

\newcommand{\C}{\mathbb{C}}
\newcommand{\A}{\mathbb{A}}
\newcommand{\Z}{\mathbb{Z}}

\newcommand{\Q}{\mathbb{Q}}

\newcommand{\mG}{\mathbb{G}}
\newcommand{\pr}{\mathbb{P}}
\renewcommand{\epsilon}{\varepsilon}

\newcommand{\D}{\mathcal{D}}
\newcommand{\M}{\mathcal{M}}

\renewcommand{\L}{\mathcal{L}}
\newcommand{\X}{\mathcal{X}}
\newcommand{\Y}{\mathcal{Y}}

\newcommand{\U}{\mathcal{U}}
\newcommand{\W}{\mathcal{W}}

\renewcommand{\phi}{\varphi}

\newcommand\Pic{\mathrm{Pic}}
\newcommand\PGL{\mathrm{PGL}}

\newcommand\cD{\mathcal{D}}

\newcommand\cL{\mathcal{L}}

\newcommand\cO{\mathcal{O}}

\newcommand\cX{\mathcal{X}}

\renewcommand{\O}{\mathrm{O}}

\makeatletter
\newcommand{\vast}{\bBigg@{4}}
\newcommand{\Vast}{\bBigg@{5}}
\makeatother

\pgfplotsset{compat=1.18}

\title{On products of K-moduli spaces}

\author{Theodoros S. Papazachariou}
\address{Isaac Newton Institute, University of Cambridge, 20 Clarkson rd, Cambridge, CB3 0EH, United Kingdom}
\email{tsp35@cam.ac.uk}
\begin{document}

\begin{abstract}
    We study the K-moduli space of products of Fano varieties in relation to the product of K-moduli spaces of the product components. We show that there exists a well-defined morphism from the product of K-moduli stacks of Fano varieties to the K-moduli stack of their product. Furthermore, we show that this morphism is an isomorphism if any two varieties with different irreducible components are non-isomorphic, and a torsor if they are. Our results rely on the theory of stacks and previous work by Zhuang. 
    
    We use our main result to obtain an explicit description of the K-moduli stack/ space of Fano threefolds with Picard rank greater than 6, along with a wall-crossing description, and a detailed polyhedral wall-crossing description for K-moduli of log Fano pairs.
\end{abstract}
\maketitle

\section{Introduction} \label{sec:intro}

In recent years, K-stability has found remarkable success in constructing projective moduli spaces for Fano varieties. In particular, due to recent innovations in moduli theory via stacks, it was shown that there exists an Artin stack $\M^K$ which parametrises K-semistable Fano varieties, or log Fano pairs. This stack admits a good moduli space $M^K$ which in turn parametrises K-polystable Fano varieties, or K-polystable log Fano pairs. The K-moduli space $M^K$ is proper and projective \cite{jiang, codogni, Li-Wang-Xu, blum_halpern-leistner_liu_xu_2021, xu_valuations, blum2021openness, xu-zhuang, xu2020uniqueness, alper_reductivity, liu2021finite}, and hence its construction answers positively the question of the existence of a moduli space of Fano varieties, which had been a long-standing open problem. We prompt the reader to the survey \cite{xu2020kstability} for an excellent exposition on this topic.

The construction of these spaces is not explicit, and their complete description has only been achieved in few examples \cite{liu-xu, liu2020kstability, spotti_sun_2017, pap22, abban2023onedimensional}. However, there exist general methods \cite{Dervan_2016,Zhuang_2020} that allow us to determine the K-stability of specific Fano varieties by relating it to the K-stability of others. In particular, it was shown by Zhuang \cite{Zhuang_2020} (see Theorem \ref{zhuang_thm}) that a Fano variety $X = X_1\times X_2$ is K-(semi/poly)stable if and only if $X_1$ and $X_2$ are K-(semi/poly)stable. 

%Although this result does not give us the full description of the K-moduli space of $X$, since in theory there can be K-polystable degenerations which do not appear as products $X_1\times X_2$, it gives us a very comprehensive tool in order to detect K-stability. 

In this paper, we the above result to show that K-semistable degenerations of products in reduced connected components of the K-moduli stack $\M^K_{n,V}$ of Fano varieties of volume $V$ and dimension $n$ are also products. We also show that the universal family over a connected component that contains a variety which is a product, also splits into products. Thus, our first main result is the following:

\begin{theorem}\label{intro_thm: optimal result}[See Theorem \ref{optimal result} and Corollary \ref{max_optimal_result}]
    Let $\M$ be a connected component of $\M^K_{n,V}$. If there exists one point $[X]\in \M$, such that $X\cong X_1\times X_2$ is a product, then so is every point of $\M$. Moreover, after passing through an {\'e}tale cover, the universal family over $\M$ splits as a fibre product.
\end{theorem}

We also show that the connected components of the corresponding K-moduli stacks $\M^K_{X_1}\times \M^K_{X_2}$ and $\M^K_{X_1\times X_2}$, that contain $X_i$ and $X_1\times X_2$, respectively,  are isomorphic when $X_1$ and $X_2$ do not have the same simple components in their product decompositions. In particular, for any $X_1$, $X_2$ we show that by taking products $X = X_1\times X_2$ there exists a well-defined map 
$$\Product\colon\M^K_{X_1}\times \M^K_{X_2}\longrightarrow\M^K_{X}.$$

We then prove the following:

\begin{theorem}\label{intro_thm1}[See Theorems \ref{main_thm} and \ref{map is etale}]
    Let $X_1\not\cong X_2$ be two K-semistable Fano varieties with different simple components in their product decompositions. Then, the map $\Product$ is an isomorphism of stacks which descends to an isomorphism of good moduli spaces. If $X_1\cong X_2$ is simple, the map $\Product$ is an $S_2$-torsor. In general, the map $\Product$ is always {\'e}tale.
\end{theorem}

% Our main result is the following:

% \begin{theorem}\label{intro_thm1}[See Theorems \ref{main_thm} and \ref{map is etale}]
%     Let $X_1\not\cong X_2$ be two K-semistable Fano varieties with different simple components in their product decompositions. Then, the map $\Product$ is an isomorphism of stacks which descends to an isomorphism of good moduli spaces. If $X_1\cong X_2$ is simple, the map $\Product$ is an $S_2$-torsor. In general, the map $\Product$ is always {\'e}tale.
% \end{theorem}

We further extend Theorem \ref{intro_thm1} by proving that the map $\Product$ defined for K-moduli stacks and K-moduli spaces of products of log Fano pairs $(X_1,c_1D_1)$ and $(X_2,c_2D_2)$ is also an isomorphism (Theorem \ref{product thm log pairs}). It should be noted, that for KSBA moduli spaces, similar results were shown in \cite{Bhatt_Ho_Patakfalvi_Schnell_2013}, relying on techniques from derived algebraic geometry. The key argument in \cite{Bhatt_Ho_Patakfalvi_Schnell_2013} is a general result on deformations of products of Deligne-Mumford stacks, with the assumption that these stacks lack infinitesimal automorphisms. In particular, they show that the deformation functor of a product of Deligne-Mumford stacks is equivalent to the product of the two corresponding deformation functors. 
% In our setting, a general result on the product of deformation functors of Artin stacks, even with certain conditions, seems infeasible, although some prior partial results do exist \cite[Proposition 4.1]{kaloghiros-petracci}. Unfortunately, this approach does not seem to extend for Fano varieties due to the existence of Fano varieties with infinitesimal automorphisms.  

% Instead, we use different methods that differ widely from \cite{Bhatt_Ho_Patakfalvi_Schnell_2013}, as they are stack theoretic in nature. In particular, for Theorem \ref{intro_thm: optimal result}, when $X_1\not\cong X_2$, we show that the map $\Product$ is finite. 
For the key statement, relating the deformation theory of K-semistable products to the deformation theories of the K-semistable product components, we use the work of Zhuang \cite[Lemma 2.17]{Zhuang_2020}. This result contains the analogue of the key points of \cite{Bhatt_Ho_Patakfalvi_Schnell_2013}, which allows us to show that the map $\Product$ is open, and hence an isomorphism by Zariski's main theorem. The above result and method only applies to connected components of the K-moduli stack $\M^K_{n,V}$.

Furthermore, by considering $\Q$-Gorenstein deformations of $\Q$-Fano varieties, we extend a result by Kaloghiros--Petracci \cite[Proposition 4.1]{kaloghiros-petracci}, as well as the key methods in \cite[Theorem 3.3]{Bhatt_Ho_Patakfalvi_Schnell_2013} to show that the $\Q$-Gorenstein  deformation functor of a product is isomorphic to the product of the $\Q$-Gorenstein deformation functors of its components (see Theorem \ref{petracci extension-full}). The $\Q$-Gorenstein deformation functor can be thought as the functor of flat deformations from the canonical stack of a scheme, which is a DM stack, and allows us to use the approach of \cite{Bhatt_Ho_Patakfalvi_Schnell_2013}. Thus, effectively show that $\Q$-Gorenstein deformations of products of klt Fanos decompose into products of  $\Q$-Gorenstein deformations of the components. We combine this result with the main method of proof of \cite[Lemma 2.17]{Zhuang_2020} to prove our main result, Theorem \ref{intro_thm: optimal result}.

The above results allow us to describe fully the K-moduli spaces Fano threefolds of Picard rank $\geq 6$ in Section \ref{sec:Fano threefolds moduli}. These are threefolds $X\times Y_{9-n}$, where $X = \pr^1$ and $Y_{9-n} = \mathrm{Bl}_{n}\pr^2$ is the del Pezzo surface of degree 1, 2, 3, 4. The Fano threefold $X\times Y_{9-n}$ is a smooth member of family \textnumero 7.1, 8.1, 9.1 and 10.1 when $n = 5$, $6$, $7$, $8$ respectively. We obtain a full description of these K-moduli spaces in relation to specific GIT quotients that have been well-studied in previous literature, using Theorem \ref{intro_thm1} and the results of \cite{mabuchi_mukai_1990,odaka_spotti_sun_2016}. We also combine our results with the results of \cite{Gallardo_2020, pap22,me_junyan_jesus} to obtain explicit K-moduli wall-crossing descriptions of log Fano pairs $(X\times Y_{9-n}, cD)$, where $D\sim -K_{X\times Y_{9-n}}$, for $n = 5$, $6$, $7$.  

Furthermore, in Section \ref{sec: pairs examples}, we provide examples of higher dimensional wall-crossings for log Fano pairs. In \cite{zhou2023shape} the author showed that for log Fano pairs $(X, \sum_{i=1}^kc_iD_i)$, where $D_i\cong_{\Q}-K_X$ and $c_i\in \Q$, there exists a finite wall-chamber decomposition of the K-moduli stack $\M^K_{X, \sum_{i=1}^kc_iD_i}$ (c.f. \cite[Theorem 1.6]{zhou2023shape}). Until now, the only example of such wall crossings is due to Fujita \cite{Fujita_2021_hyp}, who studied log Fano hyperplane arrangements. Combining the extended version of Theorem \ref{intro_thm1} for log Fano pairs, and the results of \cite{Gallardo_2018, Gallardo_2020, pap22, me_junyan_jesus}, we provide a new example of such wall-crossings. In particular, we let $X_1\subset \pr^3$ be a del Pezzo surface of degree $3$, and $X_2\subset \pr^4$ be a del Pezzo surface of degree $4$, with hyperplane sections $D_i\sim -K_{X_i}$. We consider the product $X = X_1\times X_2$, and the divisor $\mathbf{c}D = c_1D_1\boxtimes c_2D_2 = c_1p_1^*D_1+c_2p_2^*D_2$, where $p_i\colon X\rightarrow X_i$ are the projections. Then, we show the following:

\begin{theorem}\label{product_pairs intro}
For all $c_1$, $c_2$, $\mathbf{c} = (c_1,c_2)$ we have
    $$\M^K_{X, \mathbf{c}D}\cong \M^K_{X_1,D_1,c_1}\times \M^K_{X_2,D_2,c_2}\cong \M^{\GIT}_3(t(c_1))\times \M^{\GIT}_4(t(c_2))$$
    and
    $$M^K_{X, \mathbf{c}D}\cong M^K_{X_1,D_1,c_1}\times M^K_{X_2,D_2,c_2}\cong M^{\GIT}_3(t(c_1))\times M^{\GIT}_4(t(c_2)).$$
%    In particular, there are 36 total wall chamber decompositions, as illustrated in Figure \ref{fig:wall-chamber}.
\end{theorem}

Here $\M^{\GIT}_3(t(c_1))$ and $\M^{\GIT}_4(t(c_2))$ are the GIT moduli stacks parametrising GIT semistable pairs of cubics and hyperplanes in $\pr^3$, and complete intersections of two quadrics and hyperplanes in $\pr^4$ which have been studied in \cite{Gallardo_2018} and \cite{pap22}, respectively. In particular, using the results of \cite{Gallardo_2018,pap22}, which we detail in Section \ref{sec: prior results}, we obtain the explicit description of 36 ``rectangular'' wall-crossings.

\subsection{Organisation of the paper}
In Section \ref{sec:prelims} we provide expository details about K-stability, K-moduli spaces and K-moduli stacks and families of Fano varieties, as well as a structure theorem for $\Q$-Fano varieties. In Section \ref{sec:product components} we prove some properties of $\Q$-Gorenstein deformations of products of $\Q$-Fano varieties, and we prove Theorem \ref{intro_thm: optimal result}. In Section \ref{sec:product map} we prove some results regarding product families, and we construct the map $\Product$ explicitly. In Section \ref{sec: properties of prod}, we prove some general results regarding products of quotient stacks and good moduli spaces, and we prove Theorem \ref{intro_thm1}. In Section \ref{sec: examples} we provide the aforementioned examples for Fano threefolds, and log Fano pairs (Theorem \ref{product_pairs intro}).

\subsection*{Acknowledgments}
I would like to thank George Cooper, Ruadha{\'i} Dervan, Andr{\'e}s Iba{\~ n}ez N{\'u}{\~ n}ez and Andrea Petracci for advice, suggestions on how to improve this draft and very useful comments. I would especially like to thank Yuchen Liu for his many useful comments that have improved this paper and for his suggestions in strengthening the results of this paper. I would also like to thank the Isaac Newton Institute for Mathematical Sciences, Cambridge, for support and hospitality during the programme ``New equivariant methods in algebraic and differential geometry'' where work on this paper was undertaken. This work was supported by EPSRC grant EP/R014604/1.

\section{Preliminaries}\label{sec:prelims}

Throughout, we work over the field of characteristic zero $\C$.

\subsection{K-stability and K-moduli spaces}

\begin{definition}
Let $X$ be a normal projective variety, and $D$ be an effective $\Q$-divisor on $X$. Then the pair $(X,D)$ is called a \textup{log Fano pair} if $-(K_X+D)$ is an ample $\Q$-Cartier divisor. A normal projective variety $X$ is called a \textup{$\Q$-Fano variety} if $(X,0)$ is a Kawamata log terminal (\textup{klt}) log Fano pair.
\end{definition}

\begin{definition}
    Let $(X,D)$ be an n-dimensional log Fano pair, and $L$ an ample line bundle on $X$ which is $\Q$-linear equivalent to $-k(K_X+D)$ for some positive integer number $k\in Z$. Then a \textup{normal test configuration} $(\X,\D;\cL)$ of $(X,D;L)$ consists of 
\begin{itemize}
    \item a normal projective variety $\X$ with a flat projective morphism $\pi:\X\rightarrow \A^1$;
    \item a line bundle $\cL$ ample over $\A^1$;
    \item a $\mG_m$-action on the polarised variety $(\X,\cL)$ such that $\pi$ is $\mG_m$-equivariant, where $\A^1$ is equipped with the standard $\mG_m$-action;
    \item the restriction $(\X\setminus \X_0;\cL|_{\X\setminus \X_0})$ is isomorphic to $(X;L)\times (\A^1\setminus\{0\})$ $\mG_m$-equivariantly;
    \item an effective $\Q$-divisor $\D$ on $\X$ such that $\D$ is the Zariski closure of $D\times (\A^1\setminus\{0\})$ in $\X$, under the identification between $\X\setminus \X_0$ and $X\times (\A^1\setminus\{0\})$.
\end{itemize}

A test configuration is called a
\begin{enumerate}
    \item \textup{product test configuration} if $$(\X,\D;\cL)\simeq (X\times \A^1,D\times \A^1;p_1^{*}L\otimes \cO_{\X}(l\X_0))$$ for some $l\in \Z$;
    \item \textup{trivial test configuration} if it is a product test configuration and the isomorphism above is $\mathbb{G}_m$-equivariant where $X$ has the trivial $\mathbb G_m$-action;
    \item \textup{special test configuration} if $\cL\sim_{\Q}-k(K_{\X/\A^1}+\D)$ and $(\mathcal X, \mathcal D + \mathcal X_0)$ is purely log terminal (plt). %$\X_0$ is klt.
\end{enumerate}

The \textup{Donaldson-Futaki invariant} of a normal test configuration $(\X,\D;\cL)/\A^1$ is $$\DF(\X,\D;\cL):=\frac{1}{(-K_X-D)^n}\left(\frac{n}{n+1}\cdot\frac{\overline{\cL}^{n+1}}{k^{n+1}}+\frac{(\overline{\cL}^n.(K_{\overline{\X}/\pr^1}+\overline{\D}))}{k^n}\right),$$ where $(\overline{\X},\overline{\D};\overline{\cL})$ is the natural compactification of $(\X,\D;\cL)$ over $\pr^1$. 
\end{definition}

Now we define the notion of K-stability of log Fano pairs.

\begin{definition}
    A log Fano pair $(X,D)$ is called 
    \begin{enumerate}
        \item \textup{K-semistable} if  $\DF(\X,\D;\cL)\geq0$ for any normal test configuration $(\X,\D;\cL)/\A^1$ and $k\in\Q$ such that $L$ is Cartier;
        \item \textup{K-polystable} if it is K-semistable, and $\DF(\X,\D;\cL)=0$ for some test configuration $(\X,\D;\cL)$ if and only if it is a product TC.
        \item \textup{K-stable} if it is K-semistable, and $\DF(\X,\D;\cL)=0$ for some test configuration $(\X,\D;\cL)$ if and only if it is a trivial TC.
    \end{enumerate}
\end{definition}

The following prior result which relates the K-semistability of two log Fano pairs to the K-semistability of their products will be used throughout this paper. 

\begin{theorem}[{\cite[Theorem 1.1, Corollary 3.4, Proposition 4.1]{Zhuang_2020}}]\label{zhuang_thm}
    Let $X_i$ ($i=1$, $2$) be $\Q$-Fano varieties, and let $X=X_1\times X_2$. Then, $X$ is K-semistable (resp. K-polystable, K-stable) if and only if $X_i$ ($i=1$, $2$) are both K-semistable (resp. K-polystable, K-stable).

    Let $(X_i,D_i)$ ($i=1$, $2$) be log Fano pairs and let $(X,D)=(X_1\times X_2,D_1\boxtimes D_2)$. Then,  $(X,D)$ is K-semistable (resp. K-polystable K-stable) if and only if $(X_i,D_i)$ ($i=1$, $2$) are both K-semistable (resp. K-polystable, K-stable).
\end{theorem}

We will also define the notion of $\Q$-Fano families and $\Q$\emph-Gorenstein log Fano families, which are used to define the K-moduli functors we will study in this paper.

\begin{definition}\label{Q-Gor K-ss Fano family}
    A $\Q$\emph{-Fano family} is a morphism $f\colon \X\rightarrow B$ of schemes such that 
    \begin{enumerate}
        \item $f$ is projective and flat of pure relative dimension $n$ for some positive integer $n$;
        \item the geometric fibres of $f$ are $\Q$-Fano varieties;
        \item $-K_{\X/B}$ is $\Q$-Cartier and $f$-ample;
        \item $f$ satisfies Koll{\'a}r’s condition (see, {\cite[24]{kollar2009hulls}}).
    \end{enumerate}
    We call a $\Q${-Fano family} $f\colon \X\rightarrow B$ a K-semistable $\Q${-Fano family} if in addition to the above, the fibres $\X_b$ are K-semistable.
\end{definition}

% \begin{definition}\label{Q-Gor smoothable family}
    % Let $c$, $r$ be positive rational numbers. A log Fano pair $(X,cD)$ is $\Q$\emph{-Gorenstein smoothable} if there exists a $Q$-Fano family $\pi \colon \X\rightarrow C$ over a pointed smooth curve $(0\in C)$ and a relative Mumford divisor $\D$ on $\X$ over $C$ (c.f. \cite[Definition 1]{kollar2018mumford}) such that the following holds:
    % \begin{enumerate}
    %     \item $\D$ is $\Q$-Cartier, $\pi$-ample, and $\D\sim_{\pi,\Q}-rK_{\X/C}$;
    %     \item Both $\pi$ and $\pi|_{\D}$ are smooth morphisms over $C\setminus\{0\}$;
    %     \item  $(\X_0,c\D_0)\cong(X,cD)$, in particular, X has klt singularities.
%     \end{enumerate}
Let $\pi \colon \X\rightarrow B$ be a $\Q$-Gorenstein flat family of $\Q$-Fano varieties of dimension $n$, such that the anti-canonical divisor $K_{\X/B}$ is a relatively ample $\Q$-line bundle on $\X$, and $B$ is a normal base. Suppose $\D$ is an effective $\Q$-divisor on $\X$ such that every component of $\D$ is flat over $B$. If all the  fibres $(\X_b,c\D_b)$ are $\Q$-Gorenstein log Fano pairs, we call the flat family $f\colon (\X,c\D)\rightarrow B$ a $\Q$\emph{-Gorenstein log Fano family}.
% and $\D\sim_{\pi,\Q}-rK_{\X/B}$.

% \begin{definition}
%     A $\Q$\emph{-Gorenstein log Fano family} $f\colon (\X,c\D)\rightarrow B$ over a reduced scheme $B$ consists of a $\Q$-Fano family $f\colon \X\rightarrow B$ and an effective $\Q$-divisor on $\X$ such that every component of $\D$ is flat over $B$ and $\D\sim_{\pi,\Q}-rK_{\X/C}$, such that all fibres $(\X_b,c\D_b)$ are $\Q$-Gorenstein log Fano pairs.
% \end{definition}

We will also define the notion of S-equivalence.

\begin{definition}\label{def:s-eq}
    Two K-semistable $\Q$-Fano varieties $X$ and $X'$ are \emph{S-equivalent} if they degenerate to a common K-semistable log Fano pair via special test configurations.
\end{definition}

The following Lemma will be key in proving the main theorem of this paper, as it details precisely the behaviour of deformations of products of K-semistable Fano varieties. Similar results have also appeared in \cite{Li_2018_products}.

\begin{lemma}[{\cite[Lemma 2.17]{Zhuang_2020}}] \label{lem:deform Fano product}
Let $X_i$ $(i=1,2)$ be normal projective varieties and let $X=X_1\times X_2$. Let $D$ be an effective $\Q$-divisor on $X$ such that $(X,D)$ is log Fano. Let $(\cX,\cD)$ be a pair and let $\phi:\cX\to B$ be a flat projective fibration onto a smooth variety $B$ such that the support of $\cD$ does not contain any fiber of $\phi$. Let $0\in B$ and assume that $(\cX_0,\cD_0)=(\phi^{-1}(0),\cD|_{\cX_0})$ is isomorphic to $(X,D)$. Then there exists an open set in the analytic topology $0\in U\subseteq B$ and two projective morphisms $\cX_i\to U$ $(i=1,2)$ with central fibers $X_i$ such that $\cX\times_B U \cong \cX_1\times_U \cX_2$ over $U$. Moreover, both $\cX_i$ $(i=1,2)$ are uniquely determined by $\phi$ and $U$.
    % \begin{enumerate}
    %     \item There exists an open set in the analytic topology $0\in U\subseteq B$ and two projective morphisms $\cX_i\to U$ $(i=1,2)$ with central fibers $X_i$ such that $\cX\times_B U \cong \cX_1\times_U \cX_2$ over $U$. Moreover, both $\cX_i$ $(i=1,2)$ are uniquely determined by $\phi$ and $U$.
    %     % \item If $B=\A^r$ and $(\cX,\cD)$ admits a $\G_m^r$-action such that $\phi:\cX\to \A^r$ is $\mG_m^r$-equivariant, then one can take $U=\A^r$ in $(1)$ and moreover, the factors $\cX_i$ also admit $\mG_m^r$-actions making the isomorphism $\cX \cong \cX_1\times_{\A^r} \cX_2$ equivariant.
    % \end{enumerate}
\end{lemma}

We are now in a position to define the K-moduli functor/stack for log Fano pairs. The existence and properness of the K-moduli functor as an Artin stack are due to a number of recent contributions \cite{jiang, codogni, blum_halpern-leistner_liu_xu_2021, xu_valuations, blum2021openness, xu2020uniqueness, alper_reductivity, liu2021finite}.

We begin with a definition.

\begin{definition}\label{k-moduli stack def}
    Let $n$, $V$ be positive integers. The \emph{K-moduli stack of K-semistable Fano varieties of dimension $n$ and volume $V$} is the Artin stack $\M^K_{n,V}$ defined as the functor sending a reduced base $S$ to 
    \[
{\M}^K_{n,V}(S):=\left\{\X\rightarrow S\left| \begin{array}{l}\X\rightarrow S\textrm{ is a K-semistable $\Q$-Fano family, where}\\
\textrm{the fibres have dimension $n$ and volume $V$}\\ 
\end{array}\right.\right\}.
\]
\end{definition}

\begin{theorem}[{{K-moduli Theorem}}]\label{K-moduli theorem} 
Let $r\in\Q_{\geq 1}$ and $c\in(0,1/r)$ be a rational number, and $\chi$ be the Hilbert polynomial of an anti-canonically polarised  Fano variety. Consider the moduli pseudo-functor sending a reduced base $S$ to

\[
{\M}^K_{X,D,c}(S):=\left\{(\X,\D)/S\left| \begin{array}{l}(\X,c\D)/S\textrm{ is a $\Q$-Gorenstein log Fano family,}\\ \textrm{each fiber $(\X_s,c\D_s)$ is K-semistable, and }\\ \textrm{$\chi(\X_s,\cO_{\X_s}(-mK_{\X_s}))=\chi(m)$ for $m$ sufficiently divisible.}\end{array}\right.\right\}.
\]
Then there is a reduced Artin stack ${\M}^K(c)$ of finite type over $\C$ representing this moduli pseudo-functor. The $\C$-points of ${\M}^K(c)$ parameterise K-semistable $\Q$-Gorenstein log Fano pairs $(X,cD)$ with Hilbert polynomial $\chi(X,\cO_X(-mK_X))=\chi(m)$ for sufficiently divisible $m\gg 0$ and $D\sim_{\Q}-rK_X$. Moreover, the stack ${\M}^K(c)$ admits a good moduli space $\overline{M}^K(c)$, which is a reduced projective scheme of finite type over $\C$, whose $\C$-points parameterise K-polystable log Fano pairs. 
\end{theorem}

Here, pseudo-functor refers to a 2-functor which preserves composition and identities of 1-morphisms only up to coherent specified 2-isomorphism.
% \subsection{Properties of morphisms of stacks}

% In this Section, we will define properties of morphisms of stacks that will be useful later on. Let $\X$, $\Y$ be two stacks, and let $f\colon \X\rightarrow \Y$ be a morphism of stacks.

% \begin{definition}
    
% \end{definition}

\subsection{Structure theorem for $\Q$-Fano varieties}

In this section, we will prove a structure theorem for $\Q$-Fano varieties. We expect that this result is well-known to experts, but we include it for consistency. We begin with a definition.

\begin{definition}\label{simple varieties}
    A $\Q$-Fano variety is called \emph{simple} if it is not a product of lower dimensional $\Q$-Fano varieties.
\end{definition}

We are now in a position to prove a structure Theorem for $\Q$-Fano varieties.

\begin{theorem}\label{structure theorem for Fanos}
    Given any $\Q$-Fano variety $X$, there exists a unique decomposition $X=X_1\times X_2\times\dots\times X_m$ such that each $X_i$ is a simple $\Q$-Fano variety. In particular, if there exists another decomposition $X=Y_1\times Y_2\times\dots\times Y_l$ into simple $\Q$-Fano varieties, and a map $f\colon X_1\times X_2\times\dots\times X_m\rightarrow Y_1\times Y_2\times\dots\times Y_l$, the following must be true:
    \begin{enumerate}
        \item $m=l$,
        \item there exists $\sigma \in S_m$, such that $f_i\colon X_i \rightarrow Y_{\sigma(i)}$ is an isomorphism, and $f = (f_1,\dots,f_m)$.
    \end{enumerate}
    Here, $S_m$ is the symmetric group of order $m$.
\end{theorem}
\begin{proof}
    Since $X$ is $\Q$-Fano, by Kawamata-Viehweg vanishing we have $H^i(X,\cO_{X}) = 0$ for all $i > 0$. Furthermore, by the K{\"u}nneth formula, since $X=X_1\times X_2\times\dots\times X_m$ we have $H^i(X_j,\cO_{X_j}) = 0$ for all $i > 0$ and $1\leq j\leq m$. Then, by \cite[Exercise III.12.6]{Hartshorne_2010} we have $\operatorname{Pic}(X)\cong \prod_{i=1}^m \operatorname{Pic}(X_i)$.
    
    Suppose now that there exists a different decomposition $X=Y_1\times Y_2\times\dots\times Y_l$ of $X$ into simple $\Q$-Fano varieties $Y_i$. Let $M_i$ be an ample line bundle on $Y_i$, and let $\pi\colon X\rightarrow Y_i$ be the projection from $X$ to $Y_i$. Since $X$ is $\Q$-Fano, the line bundle $\pi^*M_i$ is nef and semi-ample, and for a suitable positive integer $r\in \N$,  the map $\psi_i\coloneqq\psi_{|r\pi^*M_i|}\colon X\rightarrow Y_i$ is surjective. By the above discussion, $\pi^*M_i = \sum_{j=1}^m p_j^*L_j$, where the $L_j$ are semi-ample line bundles on each factor $X_j$, with corresponding projections $p_j\colon X\rightarrow X_j$. Note, that some of the $L_j$ may vanish. Choosing $r$ large enough, we have that the map
    $$\psi_i = \psi_{|r\sum_{j=1}^m p_j^*L_j|}\colon X\rightarrow \prod_{j=1, L_j\not\sim 0}^mX_j$$
    is also surjective, which shows that $Y_i\cong\prod_{j=1, L_j\not\sim 0}^mX_j$. But, we assumed that $Y_i$ is simple for each $i$. Hence, to avoid a contradiction, we must have $m=l$, and $Y_i \cong X_{\sigma(i)}$ for some $\sigma\in S_m$, as required.
\end{proof}

% \begin{remark}
%     The key element for the proof of Theorem \ref{structure theorem for Fanos} is the Kawamata-Viehweg vanishing, which shows that $H^i(X,\cO_{X}) = 0$ for $i>0$. Since this 
% \end{remark}

\section{Product connected components of the K-moduli stack}\label{sec:product components}

In this section, we will study the local theory of deformations of $\Q$-Fano varieties, which will allow us to prove Theorem \ref{intro_thm: optimal result}. In particular, we will show that families of $\Q$-Fano varieties which are products, decompose as products themselves.

\begin{theorem}\label{optimal result}
    Let $\M$ be a connected component of $\M^K_{n,V}$ with reduced structure. If there exists one point $[X]\in \M$, such that $X\cong X_1\times X_2$ is a product, then so is every point of $\M$. Moreover, after passing through an {\'e}tale cover, the universal family over $\M$ splits as a fibre product.
\end{theorem}
\begin{proof}
% The first statement of this corollary follows directly from Theorems  \ref{main_thm} and \ref{map is etale}. We are left to prove the statement regarding the universal family.

Let $[X]\in \M$ be a point in $\M$, such that $X\cong X_1\times X_2$ is a product. Let also $\phi\colon \X\rightarrow B$ be a flat projective fibration onto a smooth variety $B$, with special fibre $\X_0\cong X$. By Lemma \ref{lem:deform Fano product} there exists an analytic open set $U$, such that $\X$ splits into a product $\X\times_{B} U\cong \X_1\times_{U} \X_2 $, where $\X_i\rightarrow B$ are both flat projective fibrations with special fibres $\X_{i,0}\cong X_i$. 
% Furthermore, we have $\X\times_{B}U \cong \X_1\times_{U}\X_2$, for two fibrations $\X_i\rightarrow B$, with projections $\X_i\rightarrow U$, and special fibres $\X_{i,0} \cong X_i$. 
Furthermore, by Kawamata-Viehweg vanishing we have $H^j(\X_b,\cO_{\X_b} ) = 0$ for all  $j>0$ and, by K{\"u}nneth formula we have $H^j(\X_i,\cO_{\X_i} ) =0$   for $i = 1,2$ and all $j>0$. Since $\M$ has reduced structure, by shrinking $U$ if necessary (see the proof of \cite[Lemma 2.17]{Zhuang_2020}) we also have $\Pic(\X) \cong H^2(\X,\Z)$ and  $\Pic(\X)\cong \Pic(\X_0)$. Since $X$ and $X_i$ are $\Q$-Fano, as in the proof of Theorem \ref{structure theorem for Fanos}, we also have $\Pic(\X) \cong \Pic(X)\cong \Pic(X_1)\times \Pic(X_2)\cong \Pic(\X_1)\times \Pic(\X_2)$. Consider the Picard schemes $\mathbf{Pic}_{\X/B}$, $\mathbf{Pic}_{\X_i/B}$, and let $\pi_i\colon X\rightarrow X_i$ be the projections. Let $M_i$ be two ample line bundles in $X_i$, respectively. By applying the Artin approximation theorem to $\mathbf{Pic}_{\X/B}$, $\mathbf{Pic}_{\X_1/B}$ and $\mathbf{Pic}_{\X_2/B}$, we see that there exists a common {\'e}tale neighbourhood $\Tilde{U}$ over the ample line bundles $M_i$ of $X_i$. Also, the extension $\L_i$ of $L_i = \pi_i^*M_i$ to $\X$ is a $\phi$-ample line bundle on $\X$, and the fibrations (over $B$) induced by the linear system $|m\L_i|$ for sufficiently large and divisible $m$,  $\psi_i = \psi_{|m\L_i|} \colon \X \rightarrow \X_i$, satisfy that $\psi_i|_{\X_0}$ is given by the projection $\X \rightarrow \X_i$ and $\psi_1\times\psi_2|_{\X_0}$ is the isomorphism $X \cong X_1\times X_2$. In particular, we have $\X \times_B \Tilde{U}\cong \X_{1}\times_{\Tilde{U}} \X_2$, as required.

% We will show that there is an open set in the {\'e}tale topology, such that the same conditions hold. Let $X_{\mathrm{cx}}$ be the site of analytically open subsets $U$ of $X$, and let $X_{\mathrm{cx}}^*$ be the site of coverings of holomorphic maps $f\colon U\rightarrow X^{\mathrm{an}}$. Let $X_{\text{\'et}}$ be the {\'e}tale site. As in \cite[\S 2]{mumford_picard}, there exists an equivalence of categories $\alpha\colon X_{\mathrm{cx}}^*\rightarrow X_{\mathrm{cx}}$, and a continuous map $\beta\colon X_{\mathrm{cx}}^*\rightarrow X_{\text{\'et}}$. By abuse of notation, we will write the corresponding open set $\alpha^{-1}(U)$ as $U$. Consider its image $\Tilde{U}\coloneqq\beta(U)$ in $X_{\text{\'et}}$, and let $\X\times_{B} \Tilde{U}$ be the corresponding fibre product. As in the proof of Lemma \ref{lem:deform Fano product}, we have $\X\times_{B} \Tilde{U} \cong \X'_1\times_{\Tilde{U}}\X'_2$, for two fibrations $\X'_i\rightarrow B$, with projections $\X_i'\rightarrow \Tilde{U}$, and special fibres $\X'_{i,0} \cong X_i$. Furthermore, by Kawamata-Viehweg vanishing we have $H^j(\X_b,\cO_{\X_b} ) = 0$ for all  $j>0$ and, by K{\"u}nneth formula we have $H^j(\X_i,\cO_{\X_i} ) =H^j(\X’_i,\cO_{\X’_i} )=0$   for $i = 1,2$ and all $j>0$. Since $\M$ has reduced structure, we also have $\Pic(\X) \cong H^2(\X,\Z)$ and  $\Pic(\X)\cong \Pic(\X_0)$. This implies that $\Pic(X_i) \cong \Pic(\X_i)\cong \Pic(\X’_i)$, and hence, this decomposition is unique. Hence, by passing through an {\'e}tale cover, the universal family over $\M$ splits as a fibre product. 

Since K-polystable degenerations are unique \cite{uniqueness_K-ps} and the K-moduli space is proper \cite{K-moduli_properness}, K-polystable degenerations of $X$ must also be a product. From the above argument, we see that in an {\'e}tale open neighbourhood, $\Tilde{U}$, all K-semistable degenerations are also products. Since $\M$ is a connected component of $\M^K_{n,V}$ with reduced structure, we see that every point in $\M$ is a product. Furthermore, by the above argument, we also see that after passing through an {\'e}tale cover, the universal family over $\M$ splits as a fibre product.

    % To show every point is a product: use openness to show that there is an etale open nbhd around [x] in which every point is a product (This is Zhuang + some extra work moving from analytic to etale topology (see Ruadhai's suggestion+ cohomological arguments and comparing Picard groups of open subsets. last part should emulate Zhuang's proofs).
    % Now from properness of K-moduli+main thm on K-ps deg K-ps deg are also products. Then above argument shows that K-ss degenerations are also products. Since $\M$ is connected first result follows. Result on universal family should follow similar to Zhuang's proof of Lemma 2.17.
\end{proof}

%%%% MAYBE THE KALOGHIROS PETRACCI APPROACH CAN EXTEND THIS FROM REDUCIBLE TO INFINITESSIMAL?? NEED TO CHECK

% \begin{remark}\label{rem: kaloghiros-petracci}
%     It should be noted that the above result could be strengthened to include infinitesimal structures on the moduli stack/space. This would be achieved via extending \cite[Proposition 4.1]{kaloghiros-petracci} to a $\Q$-Gorenstein setting, and then applying a similar proof strategy as in \cite{Bhatt_Ho_Patakfalvi_Schnell_2013}. The main difficulty with this approach is that Fano varieties may have infinitesimal automorphisms. However, we provide the Lemma extending \cite[Proposition 4.1]{kaloghiros-petracci} to the $\Q$-Gorenstein setting below.
% \end{remark}

We will now study how $\Q$-Gorenstein deformations of Fano varieties behave in products. Our approach mimics the approach of \cite{Bhatt_Ho_Patakfalvi_Schnell_2013}, using ideas from \cite[Proposition 4.1]{kaloghiros-petracci} to the $\Q$-Gorenstein setting. 

Recall that for a scheme $X$ the canonical cover $\mathfrak{X}$ is a Deligne–Mumford stack with coarse moduli space $X$ such that $\mathfrak{X}\rightarrow X$ is an isomorphism over the Gorenstein locus of $X$. We denote by $\mathbb{T}^{\mathrm{qG},i}_X$ the $i$th $\mathrm{Ext}$ group of the cotangent complex of $\mathfrak{X}$. Here, $\mathbb{T}^{\mathrm{qG},1}_X$ is the tangent space and $\mathbb{T}^{\mathrm{qG},2}_X$. is an obstruction space for the $\Q$-Gorenstein deformation functor $\operatorname{Def}^{\mathrm{qG}}_X$ of $X$.

\begin{lemma}\label{petracci extension}
    Let $X$ and $Y$ be $\Q$-Fano varieties. The natural map $$\Product_{X,Y}(A) \colon\operatorname{Def}^{\mathrm{qG}}_X\times\operatorname{Def}^{\mathrm{qG}}_Y\rightarrow \operatorname{Def}^{\mathrm{qG}}_{X\times Y} $$
    is formally smooth and induces an isomorphism on tangent spaces.
\end{lemma}
\begin{proof}
    Let $\epsilon_X\colon \mathfrak{X} \rightarrow X$, $\epsilon_Y\colon \mathfrak{Y} \rightarrow Y$ be the canonical covering stacks of $X$ and $Y$, respectively (see \cite{Bhatt_Ho_Patakfalvi_Schnell_2013}). The product map $\epsilon_X\times \epsilon_Y$ is the canonical covering stack of $X\times Y$. Let $p_X\colon \mathfrak{X}\times \mathfrak{Y}\rightarrow \mathfrak{X}$ and $p_Y\colon \mathfrak{X}\times \mathfrak{Y}\rightarrow \mathfrak{Y}$ be the two projections. Since $\epsilon_X$ and $\epsilon_Y$ are both cohomologically affine, by Kawamata-Viehweg vanishing, we have $\mathrm{H}^i(\mathcal{O}_{\mathfrak{X}})=\mathrm{H}^i(\mathcal{O}_{\mathfrak{Y}}) = 0$ for every $i > 0$, and $\mathrm{H}^0(\mathcal{O}_{\mathfrak{X}})=\mathrm{H}^0(\mathcal{O}_{\mathfrak{Y}}) = \C$. Hence, $\mathrm{R}{q_X}_*\mathcal{O}_{\mathfrak{X}\times \mathfrak{Y}} = \mathcal{O}_{\mathfrak{X}}$ and $\mathrm{R}{q_Y}_*\mathcal{O}_{\mathfrak{X}\times \mathfrak{Y}} = \mathcal{O}_{\mathfrak{Y}}$. For every $i\geq 0$ we thus have
    \begin{equation*}
        \begin{split}
            \mathrm{Ext}^i(p_X^*L_{\mathfrak{X}}, \mathcal{O}_{\mathfrak{X}\times \mathfrak{Y}} )& = \mathrm{Ext}^i(L_{\mathfrak{X}}, \mathrm{R}{q_X}_*\mathcal{O}_{\mathfrak{X}\times \mathfrak{Y}} )\\
            &= \mathrm{Ext}^i(L_{\mathfrak{X}}, \mathcal{O}_{\mathfrak{X}} )\\
            &= \mathbb{T}^{\mathrm{qG},i}_X
        \end{split}
    \end{equation*}
    and similarly 
\begin{equation*}
        \begin{split}
            \mathrm{Ext}^i(p_Y^*L_{\mathfrak{Y}}, \mathcal{O}_{\mathfrak{X}\times \mathfrak{Y}} )& = \mathrm{Ext}^i(L_{\mathfrak{Y}}, \mathrm{R}{q_Y}_*\mathcal{O}_{\mathfrak{X}\times \mathfrak{Y}} )\\
            &= \mathrm{Ext}^i(L_{\mathfrak{Y}}, \mathcal{O}_{\mathfrak{Y}} )\\
            &= \mathbb{T}^{\mathrm{qG},i}_Y.
        \end{split}
    \end{equation*}
Recall that the cotagent complex $L_{\mathfrak{X}\times \mathfrak{Y}}$ decomposes as $L_{\mathfrak{X}\times \mathfrak{Y}} = p_X^*L_{\mathfrak{X}}\oplus p_Y^*L_{\mathfrak{Y}}$, and hence 
\begin{equation*}
        \begin{split}
        \mathbb{T}^{\mathrm{qG},i}_{X\times Y}&= \mathrm{Ext}^i(L_{\mathfrak{X}\times \mathfrak{Y}}, \mathcal{O}_{\mathfrak{X}\times \mathfrak{Y}} )\\
        &= \mathrm{Ext}^i(p_X^*L_{\mathfrak{X}}, \mathcal{O}_{\mathfrak{X}\times \mathfrak{Y}} )\oplus\mathrm{Ext}^i(p_Y^*L_{\mathfrak{Y}}, \mathcal{O}_{\mathfrak{X}\times \mathfrak{Y}} )\\
            &= \mathbb{T}^{\mathrm{qG},i}_X\oplus \mathbb{T}^{\mathrm{qG},i}_Y,
        \end{split}
    \end{equation*}
    for all $i\geq 0$. Since for $i=1$ we get the usual tangent space, and for $i=2$ we get an obstruction space for the $\Q$-Gorenstein deformation functor $\operatorname{Def}^{\mathrm{qG}}$, we obtain a bijection on tangent spaces and an injection on obstruction spaces, which implies that the map $$\operatorname{Def}^{\mathrm{qG}}_X\times\operatorname{Def}^{\mathrm{qG}}_Y\rightarrow \operatorname{Def}^{\mathrm{qG}}_{X\times Y} $$
    is formally smooth, as required.
\end{proof}

\begin{remark}
    The above proof also shows that the natural product map 
    $$\mathrm{RHom}(L_X,\mathcal{O}_X)\times \mathrm{RHom}(L_Y,\mathcal{O}_Y) \rightarrow \mathrm{RHom}(L_{X\times Y},\mathcal{O}_{X\times Y})$$
    is bijective, since the map 
    $$\mathrm{Ext}^i(p_X^*L_{\mathfrak{X}}, \mathcal{O}_{\mathfrak{X}\times \mathfrak{Y}} )\times\mathrm{Ext}^i(p_Y^*L_{\mathfrak{Y}}, \mathcal{O}_{\mathfrak{X}\times \mathfrak{Y}} )\rightarrow \mathrm{Ext}^i(L_{\mathfrak{X}\times \mathfrak{Y}}, \mathcal{O}_{\mathfrak{X}\times \mathfrak{Y}} )$$
is bijective
\end{remark}

Let $\mathrm{SArt}_{\C}$ be the $\infty$-category of derived local artinian $\C$-algebras, i.e. algebras $A\in \mathrm{SArt}_{\C}$ with $\pi_0(A)$ local with residue field $\C$, and $\oplus_i\pi_i(A)$ a finite dimensional $\C$-vector space. This category provides test objects for deformation-theoretic questions in derived algebraic geometry, and we call its objects \emph{small derived algebras}.

\begin{theorem}\label{petracci extension-full}
    Let $X$ and $Y$ be $\Q$-Fano varieties. The natural map $$\Product_{X,Y} \colon \operatorname{Def}^{\mathrm{qG}}_X\times\operatorname{Def}^{\mathrm{qG}}_Y\rightarrow \operatorname{Def}^{\mathrm{qG}}_{X\times Y} $$
    is an isomorphism of functors on $\mathrm{SArt}_{\C}$.
    % \, the $\infty$-category of derived local artinian $\C$-algebras.
\end{theorem}
\begin{proof}
    We will show that the natural product map
    $$\Product_{X,Y}(A) \colon \operatorname{Def}^{\mathrm{qG}}_X(A)\times\operatorname{Def}^{\mathrm{qG}}_Y(A)\rightarrow \operatorname{Def}^{\mathrm{qG}}_{X\times Y}(A) $$
    is an equivalence of groupoids for $A \in \mathrm{SArt}_{\C}$ by working inductively on $\dim(A)$. Recall that $\mathrm{SArt}_{\C}$ is the $\infty$-category of derived local artinian $\C$-algebras. These are called \emph{small derived algebras}, which are the objects $A \in \mathrm{SArt}_{\C}$ with $\pi_0(A)$ local with residue field $\C$, and $\oplus_i  \pi_i(A)$ is a finite dimensional as $\C$-vector space. These are natural test objects for deformation-theoretic questions in derived algebraic geometry. When $\dim(A)=1$, we have $A =\C$ and there is nothing to show as both sides are reduced to points. By induction, we may assume that the desired claim is known for all $A \in \mathrm{Art}_{\C}$ such that $\dim(A)\leq n$, where $n$ is a fixed integer. Given an $\Tilde{A} \in \mathrm{Art}_{\C}$ such that $\dim(A)= n+1$, we can find a map $\Tilde{A}\rightarrow A$ with kernel $\C$ as an $A$-module, which is classified by the derivation $D_A \colon L_A \rightarrow \C[1]$. This gives a diagram
\begin{center}
    \begin{tikzcd}[column sep=large]
    \operatorname{Def}^{\mathrm{qG}}_X(\Tilde{A})\times\operatorname{Def}^{\mathrm{qG}}_Y(\Tilde{A})\arrow[r, "\Product_{X,Y}(\Tilde{A})"]\arrow[d] & \operatorname{Def}^{\mathrm{qG}}_{X\times Y}(\Tilde{A}) \arrow[d]\\
    \operatorname{Def}^{\mathrm{qG}}_X(A)\times\operatorname{Def}^{\mathrm{qG}}_Y(A)\arrow[r, "\Product_{X,Y}(A)"] & \operatorname{Def}^{\mathrm{qG}}_{X\times Y}(A) \\
    \end{tikzcd}
\end{center}
where $\Product_{X,Y}(A)$ is bijective by assumption, and we must prove that $\Product_{X,Y}(\Tilde{A})$ is also bijective. We fix flat  $\Q$-Gorenstein deformations $f \colon \X \rightarrow \operatorname{Spec}(A)$ and $f \colon \Y \rightarrow \operatorname{Spec}(A)$ of $\mathfrak{X}$ and $\mathfrak{Y}$ to $\operatorname{Spec}(A)$, where $\mathfrak{X}$ and $\mathfrak{Y}$ are the canonical covering stacks of $X$ and $Y$, respectively. Let $\pi_{f,g}\colon \X\times_{\operatorname{Spec}(A)}\Y\rightarrow \operatorname{Spec}(A)$ denote their fiber product, and let $p_{\X}\colon \X\times_{\operatorname{Spec}(A)}\Y\rightarrow \X$ and $p_{\Y}\colon \X\times_{\operatorname{Spec}(A)}\Y\rightarrow \Y$ be the projections.

By the proof of \cite[Theorem 3.3]{Bhatt_Ho_Patakfalvi_Schnell_2013}, in order to show that all fibers of $\Product_{X,Y}(\Tilde{A})$ are non-empty, i.e., if $\pi_{f,g}$  admits a deformation across $\operatorname{Spec}(A)\rightarrow\operatorname{Spec}(A')$, then the same is true for $f$ and $g$, we must show that 
% as well as by the proof of Lemma \ref{petracci extension} where it is shown that 
the maps $$\mathrm{Ext}^i(L_{\mathfrak{X}}, \mathcal{O}_{\mathfrak{X}})\rightarrow \mathrm{Ext}^i(p_X^*L_{\mathfrak{X}}, \mathcal{O}_{\mathfrak{X}\times \mathfrak{Y}} )$$ and $$\mathrm{Ext}^i(L_{\mathfrak{Y}}, \mathcal{O}_{\mathfrak{Y}})\rightarrow \mathrm{Ext}^i(p_X^*L_{\mathfrak{X}}, \mathcal{O}_{\mathfrak{X}\times \mathfrak{Y}} )$$ are isomorphisms for all $i\geq 0$. But, the above has been shown in the proof of Lemma \ref{petracci extension}

% , we deduce that all fibers of $\Product_{X,Y}(\Tilde{A})$ are non-empty, i.e., if $\pi_{f,g}$  admits a deformation across $\operatorname{Spec}(A)\rightarrow\operatorname{Spec}(A')$, then the same is true for $f$ and $g$.

By the proof of Lemma \ref{petracci extension} we have shown that the map 
$$\mathrm{Ext}^i(p_X^*L_{\mathfrak{X}}, \mathcal{O}_{\mathfrak{X}\times \mathfrak{Y}} )\times\mathrm{Ext}^i(p_Y^*L_{\mathfrak{Y}}, \mathcal{O}_{\mathfrak{X}\times \mathfrak{Y}} )\rightarrow \mathrm{Ext}^i(L_{\mathfrak{X}\times \mathfrak{Y}}, \mathcal{O}_{\mathfrak{X}\times \mathfrak{Y}} )$$
is bijective. Then, by the proof of \cite[Theorem 3.3]{Bhatt_Ho_Patakfalvi_Schnell_2013} we immediately obtain that all fibers of $\Product_{X,Y}(\Tilde{A})$ are reduced to a point, i.e.,  that all possible deformations of $\mathfrak{X}\times_{\operatorname{Spec}(A)}\mathfrak{Y}\rightarrow \operatorname{Spec}(A)$ across $\operatorname{Spec}(A)\rightarrow \operatorname{Spec}(A')$ are obtained uniquely by taking products of deformations of each factor. Combining these two results, we conclude that $\Product_{X,Y}(\Tilde{A})$ is bijective.

% QUESTION: IS THERE \textbf{ANYWHERE} WE ASSUME THE GROUPOIDS TO BE DISCRETE? NEED TO VERY CAREFULLY GO OVER THIS
\end{proof}

\begin{corollary}\label{max_optimal_result}
    % The connected component $\M$ in Theorem \ref{optimal result} does not need to be reduced, and may include infinitesimal structures.
    Let $\M$ be a connected component of $\M^K_{n,V}$. If there exists one point $[X]\in \M$, such that $X\cong X_1\times X_2$ is a product, then so is every point of $\M$. Moreover, after passing through an {\'e}tale cover, the universal family over $\M$ splits as a fibre product.
\end{corollary}
\begin{proof}
    Let $X_1$, $X_2$ be two K-polystable simple $\Q$-Fano varieties. Let $A_{X_1}$ and $A_{X_2}$ be the bases of the miniversal deformations of $X_1$ and $X_2$ respectively and let $G_{X_1}$ and $G_{X_2}$ be the automorphism groups of $X_1$ and $X_2$ respectively. 
    From Theorem \ref{petracci extension-full} the base of miniversal deformations of $X=X_1\times X_2$ is $A_{X_1}\times A_{X_2}$. If $X_1\not\cong X_2$, we have $\Aut(X)\cong G_{X_1}\times G_{X_2}$. Otherwise, $\Aut(X)\cong G_{X_1}^2\rtimes S_2$. The local structure of $\M^{K}_{n,V}$ is then given either by 
    $$[\operatorname{Spec}(A_{X_1})\times\operatorname{Spec}(A_{X_2})/ G_{X_1}\times G_{X_2}]\cong [\operatorname{Spec}(A_{X_1})/ G_{X_1}]\times[\operatorname{Spec}(A_{X_2})/ G_{X_2}]$$
    or
    $$[\operatorname{Spec}(A_{X_1})\times\operatorname{Spec}(A_{X_1})/ G_{X_1}\times G_{X_1}\rtimes S_2]\cong [[\operatorname{Spec}(A_{X_1})/ G_{X_1}^2]/S_2].$$
    In both cases, the structure may not be reduced. Let $X_1'$ and $X_2'$ be K-semistable Fano varieties, such that $X_1$ and $X_2$ are the unique K-polystable limits. For a connected component $\M$ containing $X$ and $X'=X_1'\times X_2'$, combining the results of Theorems \ref{optimal result} and \ref{petracci extension-full}  we know that $\Q$-Gorenstein deformations of $X$ and $X'$ split as products of $\Q$-Gorenstein deformations of $X_1$ and $X_2$ in a {\'e}tale neighbourhood $U$. In particular, all points in $\M$ are given by products, and we obtain the desired result.
\end{proof}

\section{Product map for K-moduli stacks}\label{sec:product map}

Recall that the K-moduli stack $\M^K_{n,V}$ is defined as follows: 

\[\mathcal{M}^{K}_{n,V}(S)\vcentcolon= \left\{ 
    \begin{aligned}
         \text{ flat proper morphisms } f\colon \mathcal{X}\rightarrow S \text{, with fibres  }\mathcal{X}_t  \text{ that are} \\
        n\text{-dimensional }K\text{-semistable } \mathbb{Q}\text{-Fano varieties with volume }
        % satisfying Koll{\'a}r's condition  (see, {\cite[24]{kollar2009hulls}})
        V
    \end{aligned}
\right\}.\] In this section, we will define reduced connected components of this K-moduli stack, using varieties $X\cong X_1\times X_2$, and we will prove the global versions of our results (i.e. Theorem \ref{intro_thm1}) by studying the structure of connected components of the K-moduli stack.

Let $X_1$, $X_2$ be two $\Q$-Fano varieties of dimensions $n_i$ and volumes $V_i$, such that the product $X\coloneqq X_1\times X_2$ is a $\Q$-Fano variety of dimension $n = n_1+n_2$ and volume $V = \binom{n_1+n_2}{n_1}V_1\cdot V_2$. We will denote by $\M^K_{X_i}$ the reduced connected component of the K-moduli stack parametrising K-semistable (possibly singular) Fano varieties, which degenerate to the general member $X_i$. In essence $\M^K_{X_i}$ is the reduced connected component of $M_{n_i,V_i}^K$ containing $X_i$. Similarly, we denote by $\mathcal{M}^K_{X}$ the reduced connected component of the K-moduli stack parametrising K-semistable (singular) Fano varieties, which degenerate to the general member $X$. In this section, we will explicitly construct a map 
$$\Product\colon \mathcal{M}^K_{X_1}\times \mathcal{M}^K_{X_2}\rightarrow \mathcal{M}^K_{X}.$$

There is a similar expression for the K-moduli stack $\M^K_{X_i, D_i, c_i}$, which generalises the above, using the notion of $\Q$-Gorenstein log Fano families.

Let us give a more formal definition to the above notions. Let $\chi_i$ be the Hilbert polynomial of elements of the family of deformations $\operatorname{Def}(X_i)$, which is pluri-anticanonically embedded by $-m_iK_{X_i}$ in $\mathbb P^{N_i}$, for some $N_i$, and let $\mathbb H^{X_i; N}\coloneqq \mathrm{Hilb}_{\chi_i}(\mathbb P^{N_i})$. Given a closed subscheme $X_i \subset \mathbb P^{N_i}$ with Hilbert polynomial $\chi(X_i, \mathcal O_{\mathbb P^{N_i}}(k)|_{X_i})=\chi_i(k)$, let $\mathrm{Hilb}(X_i)\in \mathbb H^{\chi_i; N_i}$ denote its Hilbert point.

The Artin stack $\M^K_{X_i}$ is defined as follows: 

% \[\mathcal{M}^{K}_{X_i}(S)\vcentcolon= \left\{ 
%     \begin{aligned}
%          \text{ flat proper morphisms } f\colon \mathcal{X}\rightarrow S\in \operatorname{Def}(X) \\ \text{, with fibres  }\mathcal{X}_t  \text{ that are} 
%         n\text{-dimensional }K\text{-semistable } \\\mathbb{Q}\text{-Fano varieties} 
%         \text{satisfying Koll{\'a}r's condition  (see, {\cite[24]{kollar2009hulls}})}
%     \end{aligned}
% \right\}.\]

\[
{\M}^K_{X_i}(S):=\left\{\X/S\left| \begin{array}{l}\X/S\textrm{ is a $\Q$-Fano family, } \textrm{each fiber $\X_s $}\textrm{ is K-semistable, and }\\ \textrm{$\chi(\X_s,\cO_{\X_s}(-mK_{\X_s}))=\chi_i(m)$} \textrm{ for $m$ } \textrm{sufficiently divisible.}\end{array}\right.\right\}.
\]

The condition $\chi(\X_s,\cO_{\X_s}(-mK_{\X_s}))=\chi_i(m)$, ensures that all fibres of the $\Q$-Fano family $f\colon \X\rightarrow S$ are K-semistable degenerations of the Fano variety $X_i$.

Let
\begin{equation}\label{eq: zi}
    \hat Z_{i,m_i}\coloneqq\left\{ \mathrm{Hilb}(X_i)\in \mathbb H^{\chi_i; N_i} \;\middle|\; 
\begin{aligned}
  & X_i \text{ is a Fano manifold of dimension } n_i \text{, volume } V_i ,\\
  &\mathcal O_{\mathbb{P}^{N_i}}(1)|_{X_i}\sim \mathcal O_{X_i}(-m_iK_{X_i}),\\
  &\text{and }H^0(\mathbb P^{N_i}, \mathcal O_{\mathbb P^{N_i}}(1))\xrightarrow{\cong} H^0(X_i, \mathcal O_{X_i}(-m_iK_{X_i})).
  \end{aligned}
\right\}
\end{equation}
which is a locally closed subscheme of $\mathbb H^{\chi_i; N_i}$. Let $\overline Z_{i,m_i}$ be its Zariski closure in  $\mathbb H^{\chi_i; N_i}$ and $Z_{i,m_i}$ be the subset of $\hat Z_{i,m_i}$ consisting of K-semistable varieties, which are independent of $m$ if $m\gg 0$. By \cite{Li-Wang-Xu} $$\M^K_{X_i}  = [Z_{i}^{\mathrm{red}}/ \PGL(N_{i}+1)],$$
where $Z_i^{\mathrm{red}}$ is the reduced scheme supported on $Z_i$.

Similarly, for $X = X_1\times X_2$, let $\chi$ be the Hilbert polynomial of elements of the family of deformations $\operatorname{Def}(X)$, which is pluri-anticanonically embedded by $-mK_{X}$ in $\mathbb P^{N}$, for some $N$, and let $\mathbb H^{X; N}\coloneqq \mathrm{Hilb}_{\chi}(\mathbb P^{N})$. Given a closed subscheme $X \subset \mathbb P^{N}$ with Hilbert polynomial $\chi(X, \mathcal O_{\mathbb P^{N}}(k)|_{X})=\chi(k)$, let $\mathrm{Hilb}(X)\in \mathbb H^{\chi; N}$ denote its Hilbert point. Then,

$$\M^K_{X}  = [Z^{\mathrm{red}}/ \PGL(N+1)],$$
    where 

    \begin{equation}\label{z:def for prod}
        \hat Z_m\coloneqq\left\{ \mathrm{Hilb}(X)\in \mathbb H^{\chi; N} \;\middle|\; 
\begin{aligned}
  & X\text{ is a Fano manifold of dimension } n_1+n_2,\\
  & \text{volume } V, \mathcal O_{\mathbb{P}^{N}}(1)|_{X}\sim \mathcal O_{X}(-mK_{X}) \text{, and}\\
  &H^0(\mathbb P^{N}, \mathcal O_{\mathbb P^{N}}(1))\xrightarrow{\cong} H^0(X, \mathcal O_{X}(-mK_{X})),
  \end{aligned}
\right\}
    \end{equation}
which is a locally closed subscheme of $\mathbb H^{\chi; N}$. Let $Z_m$ be the subset of $\hat Z_m$ consisting of K-semistable varieties in $\hat Z_m$, and let $Z^{\mathrm{red}}$ be the reduced scheme supported on $Z_m$, which are independent of $m$ if $m\gg 0$.

% Furthermore, by Lemma \ref{products of quotient stacks} we have 
% $$\M^K_{X_1}\times \M^K_{X_2}\cong  \big[(Z_{1}^{\mathrm{red}}\times Z_{2}^{\mathrm{red}})/ \big(\PGL(N_{1}+1)\times \PGL(N_{2}+1)\big)\big].$$

\begin{lemma}\label{products are klt}
    Let $X_1$, $X_2$ be two varieties with klt singularities. Then, the product $X_1\times X_2$ also has klt singularities.
\end{lemma}
\begin{proof}
    Let $X_1$, $X_2$ be as above, and let $X \coloneqq X_1\times X_2$. By assumption, $X_1$ and $X_2$  are reduced and $\Q$-Gorenstein, and the same is true for $X$. Let $f\colon Y_i\rightarrow X_i$ be log resolutions for the two varieties. Then we have
    $$K_{Y_i} = f_i^*K_{X_i} + \sum_j \alpha_{i,j}E_{i,j}$$
    where $\alpha_{i,j}>-1$ since $X_i$ are klt. Let $Y \coloneqq Y_1\times Y_2$ and $f = f_1\times f_2$; then, the morphism $f\colon Y\rightarrow X$ is a log resolution of $X$ and we have 
    $$K_Y = p_1^*K_{Y_1}+p_2^*K_{Y_2} = f^*(K_{X_1}+K_{X_2}) + \sum_j(\alpha_{1,j}E_{1,j}\times Y_2+\alpha_{2,j}Y_1\times E_{2,j}),$$
    where $p_i\colon Y\rightarrow Y_i$ are the projection morphisms. Since $\alpha_{i,j}>-1$, and the discrepancy does not depend on the log resolution, $X$ also has klt singularities.
\end{proof}

We will now show that a product of $\Q$-Gorenstein K-semistable log Fano families, is also a $\Q$-Gorenstein K-semistable log Fano family.

\begin{proposition}\label{products of families}
    Let $f_i\colon \X_i \rightarrow B$ be 
    %smoothable $\Q$-Gorenstein K-semistable log Fano families,
    K-semistable $\Q$-Fano families for $i = 1$, $2$. Then, the family $f \colon \X_1\times_B\X_2\rightarrow B$ is also a K-semistable $\Q$-Fano family.
    %$\Q$-Gorenstein K-semistable log Fano family.
\end{proposition}
\begin{proof}
    Let $\Y\coloneqq \X_1\times_B\X_2$. Since $f_i$ is a K-semistable $\Q$-Fano family, 
%    $\Q$-Gorenstein K-semistable log Fano family, 
for $i = 1$, $2$, the fibres $\X_{i,b}$ are K-semistable Fano varieties with klt singularities. By Lemma \ref{products are klt}, we see that the product $\X_{1,b}\times \X_{2,b}$ also has klt singularities. Furthermore, by Theorem \ref{zhuang_thm} since $\X_{i,b}$ is K-semistable for $i = 1$, $2$, the product $\Y_b\coloneqq \X_{1,b}\times_B\X_{2,b}$ is also K-semistable. By
\cite[Proposition 2.12]{Bhatt_Ho_Patakfalvi_Schnell_2013} and \cite[Lemma 7.3]{Kovacs_2009}, the sheaf $\omega^{[k]}_{\Y/B}$ commutes with arbitrary base change, and hence it satisfies Koll{\'a}r's condition. Hence, $f$ is a K-semistable $\Q$-Fano family.
\end{proof}

Let $f_i\colon (\X_i,c_i\D_i)\rightarrow B$, for $i=1$, $2$ be two $\Q$-Gorenstein log Fano families. As in Proposition \ref{products of families} we can define the fibre product family $f\colon (\X_1\times_B\X_2, \mathbf{c}\D )\rightarrow B$, where $\mathbf{c}\D \coloneqq c_1D_1\boxtimes c_2D_2 = \sum_{i=1}^2p_i^*(c_i\D_i)$, and $p_i\colon \X_1\times_B\X_2\rightarrow \X_i$ are the two projections. 
% We will show that $f$ is also a $\Q$-Gorenstein log Fano family.

\begin{remark}\label{rmk: products of Q-gor smooth families}
    The same method of proof as in Proposition \ref{products of families} shows that products of $\Q$-Gorenstein log Fano families are also $\Q$-Gorenstein log Fano families, but we omit the proof.
\end{remark}

% \begin{proposition}\label{products of Q-gor smooth families}
%     Let $f_i\colon (\X_i,c_i\D_i)\rightarrow B$, for $i=1$, $2$ be two $\Q$-Gorenstein smoothable log Fano families with K-semistable fibres. Then $f\colon (\X_1\times_B\X_2, \mathbf{c}\D )\rightarrow B$ is also a $\Q$-Gorenstein smoothable log Fano family with K-semistable fibres.
% \end{proposition}
% \begin{proof}
%     $f$ is a $\Q$-Fano family from Proposition \ref{products of families}. Since the $f_i$ are $\Q$-Gorenstein smoothable log Fano families with K-semistable fibres, the fibres $(\X_{i,b},c_i\D_{i,b})$ are K-semistable, and hence the log Fano pair $(\X_{b}, \mathbf{c}\D_{b})$ is K-semistable by Theorem \ref{zhuang_thm}. Since each fibre $(\X_{i,b},c_i\D_{i,b})$ is a $\Q$-Gorenstein smoothable log Fano pair, $\mathbf{c}\D_b = c_1p_1^*\D_{1,b}+c_2p_2^^*\D_{2,b} \sim_{f,\Q} -r_1p_1^*K_{\X_{1,b}/B}-r_2p_2^*K_{\X_{2,b}/B}\sim_{f,\Q} -rK_{\X/B}$. Furthermore, by Lemma \ref{products are klt}, $\X_{0}$ has klt singularities, and $f$, $f|_{\D_b}$ are smooth morphisms over $B\setminus \{0\}$. Hence, $f$ is a $\Q$-Gorenstein smoothable log Fano family with K-semistable fibres.
% \end{proof}

\begin{definition}\label{product map def}
    Let $X_i$ be two Fano varieties of dimensions $n_i$ and volumes, $V_i\in \N$. The product map 
    $$\Product\colon \M^K_{X_1}\times \M^K_{X_2}\longrightarrow \M^K_{X},$$
where $X$ has dimension $n_1+n_2$ and volume $V = \binom{n_1+n_2}{n_1}V_1\cdot V_2$, 
    % For any two K-semistable Fano varieties $X_1$, $X_2$, the map 
    % $$\Product\colon \mathcal{M}^K_{X_1}\times \mathcal{M}^K_{X_2}\longrightarrow \mathcal{M}^K_{X}$$
    is defined by taking fibre products of K-semistable $\Q$-Fano families.
    The product map also descends to a product map $\overline{\Product}$ on the K-moduli spaces 
    $$\overline{\Product}\colon M^K_{X_1}\times M^K_{X_2}\longrightarrow M^K_{X}.$$
\end{definition}

% Key question here: maybe we need to briefly amend definition of rhs stack, to reflect we ONLY consider products?

We will similarly define the product map for log Fano pairs:

\begin{definition}\label{product map pairs}
    Let ${\M}^K_{X_i, D_i, c_i}$ be two K-moduli stacks for $i=1$, $2$. The product map 
    $$\Product_{c_1,c_2}\colon {\M}^K_{X_1, D_1, c_1}\times {\M}^K_{X_2, D_2, c_2}\longrightarrow {\M}^K_{X_1\times X_2, D, \mathbf{c}}$$
    is defined by taking fibre products of $\Q$-Gorenstein log Fano families with K-semistable fibres. The map descends to a map 
    $$\overline{\Product}_{c_1,c_2}\colon {M}^K_{X_1, D_1, c_1}\times {M}^K_{X_2, D_2, c_2}\longrightarrow {M}^K_{X_1\times X_2, D, \mathbf{c}}$$
    at the level of good moduli spaces.
\end{definition}

\begin{remark}\label{def of map via Zhuang}
    At the level of $\C$-points and pseudofunctors, as in Theorem \ref{K-moduli theorem}, this map is defined explicitly as 
    \begin{equation*}
        \begin{split}
           \Product\colon \M^K_{X_1}\times \M^K_{X_2}&\longrightarrow \M^K_{X}\\
            ([Y],[Z])&\longmapsto [Y\times Z],  
        \end{split}
    \end{equation*}
    where $Y\times Z$ is K-semistable by Theorem \ref{zhuang_thm}. The map on K-moduli spaces is similarly defined by sending $([Y],[Z])\mapsto [Y\times Z]$, where again $Y\times Z$ is K-polystable by Theorem \ref{zhuang_thm}. The map on K-moduli stacks and K-moduli spaces of log pairs is also similarly defined.
\end{remark}

% General comment for me: Maybe the definition of the map above should be more general? Instead of defining it in terms of $X_1$ and $X_2$, we can define it for any two arbitrary K-moduli stacks, since we have shown a product of K-ss $\Q$-Fano families is K-ss $\Q$-Fano family, AND a product of Fanos is always Fano so this is well-defined. Then define the map as 

\section{Properties of the product map}\label{sec: properties of prod}
The first result regarding products of quotient stacks is well known to experts.

\begin{proposition}\label{products of quotient stacks}
    Let $\X \coloneqq[X/G]\rightarrow S$, $\Y\coloneqq [Y/H]\rightarrow S$ be two quotient stacks. Then, $\X\times_S\Y\cong [X\times Y/(G\times H)]$ is also a quotient stack.
\end{proposition}

We will now show that for two quotient stacks $\X$ and $\Y$, admitting good moduli spaces $\pi_{\X}\colon\X\rightarrow X$ and $\pi_{\Y}\colon\Y\rightarrow Y$ respectively, the product $X\times_S Y$ of the good moduli spaces is in fact the good moduli space of the fibre product of the stacks $\X\times_S \Y$. The below statement is also shown in \cite[Lemma 4.15]{Alper_good}, but we provide an alternative proof.

\begin{proposition}\label{product of good moduli is good moduli}
    Let $\X$ and $\Y$ be two quotient stacks with good moduli spaces $\pi_{\X}\colon \X\rightarrow X$ and $\pi_{\Y}\colon\Y\rightarrow Y$ respectively. Then the product $\X\times_S\Y$ is a quotient stack that admits good moduli space $X\times_S Y$.
\end{proposition}
\begin{proof}
From Proposition \ref{products of quotient stacks}, we see that the stack $\X\times_S\Y$ is a quotient stack. Consider the following commutative cube diagram
\begin{center}
    \begin{tikzcd}[row sep=2.5em]
\X\times_S\Y \arrow[rr,"pr_2"] \arrow[dr,swap,"a"] \arrow[dd,swap,"pr_1"] &&
  \Y \arrow[dd,swap, near start] \arrow[dr,"\pi_{\Y}"] \\
& X\times_S Y \arrow[rr,crossing over,"pr_2" near start] &&
  Y \arrow[dd] \\
\X \arrow[rr, near end] \arrow[dr,swap,"\pi_{\X}"] && S \arrow[dr,swap,"="] \\
& X \arrow[rr] \arrow[uu,<-,crossing over,"pr_1" near end]&& S,
\end{tikzcd}
\end{center}
where $pr_i$ are the corresponding projections, $\pi_{\X}$ and $\pi_{\Y}$ are the good moduli space maps, and the map $a$ exists due to the universal property of fibre products. Furthermore, the back and front faces of this diagram are Cartesian, and since $pr_i$, $\pi_{\X}$ and $\pi_{\Y}$ are surjective, so is $a$.

Considering the left-hand side face of the diagram, and due to the base change property of good moduli spaces, we have that $\X\times_X (X\times_S Y) \rightarrow X\times_S Y$ is a good moduli space. Similarly, by considering the top face of the diagram, we have that $\Y\times_Y (X\times_S Y) \rightarrow X\times_S Y$ is also a good moduli space. Furthermore, since the back face of the diagram is Cartesian and since $a$ is surjective, we have $\X\times_X (X\times_S Y)\cong \Y\times_Y (X\times_S Y) \cong \X \times_S \Y$. Hence, the map $a: \X \times_S \Y\rightarrow X\times_S Y$ is a good moduli space.
\end{proof}

The above proposition, along with Definition \ref{product map def} shows that there exists a commutative diagram 
\begin{center}
    \begin{tikzcd}
    \M^K_{X_1}\times \M^K_{X_2}\arrow[r, "\Product"]\arrow[d] & \M^K_{X}\arrow[d]\\
    M^K_{X_1}\times M^K_{X_2}\arrow[r, "\overline{\Product}"] & M^K_{X} 
    \end{tikzcd}
\end{center}
where and $M^K_{X_1}$,  $M^K_{X_2}$ and $M^K_{X}$ are the K-moduli spaces of the K-moduli stacks $\M^K_{X_1}$, $\M^K_{X_2}$ and $\M^K_{X}$, respectively. The $\Product$ map at the level of good moduli spaces is defined by  $([Y],[Z])\mapsto [Y\times Z]$, where $Y$, $Z$ are K-polystable Fano varieties.
% of dimensions $n$ and $m$ and volumes $V_1$, $V_2$, respectively.

% Point for me here. There doesn't seem much to prove for me here... but let's carry on. 

% fix below accordingly

% Let $\chi_i$ be the Hilbert polynomial of smooth elements of the family of deformations $\operatorname{Def}(X_i)$, which is pluri-anticanonically embedded by $-m_iK_{X_i}$ in $\mathbb P^{N_i}$, for some $N_i$, and let $\mathbb H^{X_i; N}\coloneqq \mathrm{Hilb}_{\chi_i}(\mathbb P^{N_i})$. Given a closed subscheme $X_i \subset \mathbb P^{N_i}$ with Hilbert polynomial $\chi(X_i, \mathcal O_{\mathbb P^{N_i}}(k)|_{X_i})=\chi_i(k)$, let $\mathrm{Hilb}(X_i)\in \mathbb H^{\chi_i; N_i}$ denote its Hilbert point. Let
% $$\hat Z_{i,m}\coloneqq\left\{ \mathrm{Hilb}(X_i)\in \mathbb H^{\chi_i; N_i} \;\middle|\; 
% \begin{aligned}
%   & X_i \text{ is a Fano manifold of dimension } n_i \text{, volume } V_i ,\\
%   &\mathcal O_{\mathbb{P}^{N_i}}(1)|_{X_i}\sim \mathcal O_{X_i}(-m_iK_{X_i}),\\
%   &\text{and }H^0(\mathbb P^{N_i}, \mathcal O_{\mathbb P^{N_i}}(1))\xrightarrow{\cong} H^0(X_i, \mathcal O_{X_i}(-m_iK_{X_i})).
%   \end{aligned}
% \right\}$$
% which is a locally closed subscheme of $\mathbb H^{\chi_i; N_i}$. Let $\overline Z_{i,m}$ be its Zariski closure in  $\mathbb H^{\chi_i; N_i}$ and $Z_{i,m}$ be the subset of $\hat Z_{i,m}$ consisting of K-semistable varieties. For $m\gg >0$, the choice of $m$ is irrelevant, hence we will write $Z_i$. By \cite{Li-Wang-Xu} 
Recall that
$$\M^K_{X_i}  = [Z_{i}^{\mathrm{red}}/ \PGL(N_{i}+1)],$$
where $Z_i^{\mathrm{red}}$ is the reduced scheme supported on $Z_i$, which is defined in Equation \eqref{eq: zi}. Then, by Lemma \ref{products of quotient stacks} we have 
$$\M^K_{X_1}\times \M^K_{X_2}\cong  \big[(Z_{1}^{\mathrm{red}}\times Z_{2}^{\mathrm{red}})/ \big(\PGL(N_{1}+1)\times \PGL(N_{2}+1)\big)\big].$$

% \begin{proposition}
%     The map $\Product$ is {\'e}tale.
% \end{proposition}
% \begin{proof}
%     dfefefe
% \end{proof}

\begin{theorem}\label{main_thm}
    Let $X_1\not\cong X_2$ be two K-semistable Fano varieties with different simple components in their product decompositions. Then the map $\Product$ is an isomorphism of stacks which descends to an isomorphism of good moduli spaces. In particular, $\M^K_{X_1}\times \M^K_{X_2}\cong \M^K_{X}$ and $M^K_{X_1}\times M^K_{X_2}\cong M^K_{X}$.
\end{theorem}
\begin{proof}
    We will show that $\Product$ is an isomorphism. By Zariski’s Main Theorem of stacks, it suffices to show that $\Product$ is a finite open immersion. We will prove $\Product$ is finite using \cite[Proposition 6.4]{Alper_good}, which requires showing that the morphism on the level of good moduli spaces is finite,and that $\Product$ is representable, separated, quasi-finite, and sends closed points to closed points. Notice that, for two closed points $[x_i]\in \M^K_{X_i}$, the point $[x_1\times x_2] \in  \M^K_{X_1}\times  \M^K_{X_2}$ is closed by Theorem \ref{zhuang_thm}, hence the map $\Product$ sends closed points to closed points. 

    To show that the map $\Product$ is separated we need to show that two pairs of families $(f_1\colon \X_1\rightarrow B, f_2\colon \X_2\rightarrow B)$ and $(f'_1\colon \X'_1\rightarrow B, f'_2\colon \X'_2\rightarrow B)$ which agree over the general fiber, i.e. $(\X_{1,t}, \X_{2,t})\cong (\X'_{1,t}, \X'_{2,t})$ for $t\neq 0$, and have the same image in $\M^K_{X}$, degenerate to the same closed fiber $(\X_{1,0}, \X_{2,0})$ in $\M^K_{X_1}\times \M^K_{X_2}$. Recall that $M^K_{X_i}$ is separated, and that since $f_i\colon \X_i\rightarrow B$ and $f'_i\colon \X'_i\rightarrow B$ are families such that $\X_{i,t}\cong \X'_{i,t}$, the families degenerate to the same closed fibre $\X_{i,t}\cong \X'_{i,0}$ in $\M^K_{X_i}$. Furthermore, since $\Product$ sends closed points to closed points, the fibre $(\X_{1,t}, \X_{2,t})\cong (\X'_{1,t}, \X'_{2,t})$ is closed in $M^K_{X}$. Let $\X = \X_1\times \X_2$ and $\X' = \X_1'\times \X_2'$. As in the proof of Theorem \ref{optimal result} we have $\Pic(\X)\cong \Pic(\X')\cong \Pic(X)$. Let $\L$, $\L'$ be ample line bundles in $\X$ and $\X'$ respectively. By the above argument we can identify $\L\cong \L'$, and hence 
    $$\X \cong \operatorname{Proj}\bigoplus_{k}\mathrm{H}^0(\X,\L^k)\cong \operatorname{Proj}\bigoplus_{k}\mathrm{H}^0(\X',(\L')^k)\cong \X'.$$
    Hence, both families $f_1\times f_2$ and $f'_1\times f'_2$ degenerate to the same closed fibre, and are isomorphic, thus $\Product$ is separated.
    %%% todo: do I need to say anything more about why families are the same here?? Does it not follow immediatelly that the families are isomorphic? Need to check

    We next tackle representability. For this, we will show that the map on stabilisers at the closed points is injective. The map is given by $f\colon (\Aut(Y), \Aut(Z)) \rightarrow \Aut(Y\times Z)$, where $[Y]$ is a closed point in $\M^K_{X_1}$ and $[Z]$ is a closed point in $\M^K_{X_2}$. Since we assume $X_1\not\cong X_2$ the map $f$ is injective. Furthermore, since $X_1\not\cong X_2$ (i.e. we do not allow pairs $([Y],[Z])\not\cong ([Z], [Y])$ such that they both map to $[Y\times Z]\cong [Z\times Y]$) the map $\Product$ is injective by construction. To see this more clearly, suppose that $([Y],[Z]), ([Y'],[Z'])\in \M^K_{X_1}\times \M^K_{X_2}$ are two points such that $$\Product(([Y],[Z])) = [Y\times Z] \cong [Y'\times Z'] = \Product(([Y'],[Z'])).$$
    This implies that $Y\times Z$ and $Y'\times Z'$ are $S$-equivalent, i.e. they degenerate to a common K-semistable Fano variety $A$ via special test configurations. Suppose that $([Y],[Z])\not\cong ([Y'],[Z'])$; then, either $[Y]\not \cong [Y']$ or $[Z]\not \cong [Z']$ (or both). Let's assume that $[Y]\not \cong [Y']$ but $[Z]\cong [Z']$; then $Y$ and $Y'$ are not S-equivalent, i.e. they degenerate to different K-semistable Fano varieties $A_Y$ and $A_{Y'}$, of dimension $n_1$ and  volume $V_1$, respectively, while $Z$ and $Z'$ are S-equivalent, i.e., they degenerate to a common K-semistable Fano variety $A_Z$ of dimension $n_2$ and  volume $V_2$. Then, by Theorem \ref{zhuang_thm}, $Y\times Z$ degenerates to the K-semistable Fano variety $A_Y\times A_Z$ via special test configurations, while $Y'\times Z'$ degenerates to the K-semistable Fano variety $A_{Y'}\times A_{Z}$ via special test configurations, where $A_Y\times A_Z\not\cong A_{Y'}\times A_{Z}$. But, this contradicts the assumption $[Y\times Z] \cong [Y'\times Z']$. Hence, $\Product$ is injective. Furthermore, the descended map on the good moduli spaces $\overline{\Product}$ is also injective. This argument, along with Lemma \ref{lem:deform Fano product}, Theorem \ref{structure theorem for Fanos} and Theorem \ref{optimal result}, shows that $\Product$ is an open immersion of stacks, and $\overline{\Product}$ is an open immersion of schemes. Since we assume $X_1$ and $X_2$ to have different simple components in their product decompositions, and small deformations of products of K-semistable Fano type varieties is still a product. Lemma \ref{lem:deform Fano product} proves openness in the analytic topology, but we note that the map is an open immersion of stacks in general, via [Cor. 2.3, Exp. XII]\cite{grothendieck}.
    The above argument also shows that $\Product$ is an {\'e}tale morphism of stacks, since it is an open immersion.

    % (QUESTION: DO THESE DEGENERATIONS HAVE TO BE UNIQUE? WE SHOULD CHECK)

    We will now show quasi-finiteness. Recall that 
    $$\M^K_{X}  = [Z^{\mathrm{red}}/ \PGL(N+1)],$$
    where $Z^{\mathrm{red}}$ is defined in Equation \eqref{z:def for prod}.
%     where 
%     $$\hat Z\coloneqq\left\{ \mathrm{Hilb}(X\times Y)\in \mathbb H^{\chi_1\times\chi_2; N = N_1+N_2} \;\middle|\; 
% \begin{aligned}
%   & X\times Y \text{ is a Fano manifold of dimension } n_1+n_2,\\
%   & \text{volume } V, \mathcal O_{\mathbb{P}^{N}}(1)|_{X\times Y}\sim \mathcal O_{X\times Y}(-mK_{X\times Y}) \text{, and}\\
%   &H^0(\mathbb P^{N}, \mathcal O_{\mathbb P^{N}}(1))\xrightarrow{\cong} H^0(X\times Y, \mathcal O_{X\times Y}(-mK_{X\times Y})),
%   \end{aligned}
% \right\}$$
% and $Z$ is the subset of $\hat Z$ consisting of K-semistable varieties with $Z^{\mathrm{red}}$ the reduced scheme supported on $Z$. 
Consider the morphisms $Z_i^{\mathrm{red}} \rightarrow [Z_i^{\mathrm{red}}/ \PGL(N_i+1)]$, and $\PGL(N_i+1)$-torsors $P_i \rightarrow Z_i^{\mathrm{red}}$, whose fibres over $X_i \in Z_i^{\mathrm{red}}$ are $(X_i, [s_0,\dots, s_{N_i}])$, where the $s_j$ are a basis for $H^0(X_i, \cO_{X_i}(-m_iK_{X_i}))$. Consider $P\coloneq P_1\times P_2$. Then, there exists a $\PGL(N_1+1)\times \PGL(N_2+1)$-equivariant morphism $P\rightarrow Z^{\mathrm{red}}$. We have the following diagram:

\begin{center}
        \begin{tikzcd}
        P \arrow[r, "g"]\arrow[d]\ar[rr,bend left=35] &\left[P/\big(\PGL(N_1+1)\times \PGL(N_2+1)\big)\right]\arrow[d]\arrow[r, "f"] &Z^{\mathrm{red}}\arrow[d]\\
        Z_1^{\mathrm{red}}\times Z_2^{\mathrm{red}} \arrow[r] & \left[Z_1^{\mathrm{red}}\times Z_2^{\mathrm{red}}/ \big(\PGL(N_1+1)\times \PGL(N_2+1)\big)\right] \arrow[r, "\Product"]&  \left[Z^{\mathrm{red}}/ \PGL(N+1)\right].
        \end{tikzcd}
    \end{center}

We will show that the map $f$ is quasi-finite, which in turn will imply that $\Product$ is quasi-finite, since quasi-finiteness can be shown {\`e}tale locally.   

Let $p_j \coloneqq (X_j, [s_0^{j,1},\dots, s_{N_1}^{j,1}])$, where $j=1$, $2$, be two points in $P_1$ and $r_j \coloneqq (Y_j, [s_0^{j,2},\dots, s_{N_2}^{j,2}])$ be two points in $P_2$, such that $p_j\times r_j$ are two distinct points in $P$. Let $\overline{p_j\times r_j}$ be their corresponding images inside $\left[P/\big(\PGL(N_1+1)\times \PGL(N_2+1)\big)\right]$. Assume that the images of $\overline{p_j\times r_j}$ inside $Z^{\mathrm{red}}$ under $f$ are the same. This implies that we have an isomorphism $h\colon X_1\times Y_1\rightarrow X_2\times Y_2$ such that $X_1\times Y_1$ and $X_2\times Y_2$ belong in the same $\PGL(N_1+1)\times \PGL(N_2+1)$-orbit, and $h^*(s_k^{2,1}\times s_k^{2,2}) = s_k^{1,1}\times s_k^{1,2}$. But, this implies that there are finitely many preimages of a point in $Z^{\mathrm{red}}$ under $f$ and so $f$ is quasi-finite. By the above discussion, it follows that $\Product$ is also quasi-finite. Furthermore, by descent to good moduli spaces, the descended map $\overline{\Product}$ is also quasi-finite.

We are left to check that the map $\overline{\Product}$ on good moduli spaces is finite. But, notice that by Theorem \ref{K-moduli theorem} all three $M^K_{X_i}$ and $M^K_{X}$ are proper schemes, which implies that $M^K_{X_1}\times M^K_{X_2}$ is also proper. Hence, the map $\overline{\Product}$ is a proper map. By the previous argument, it is also quasi-finite, and hence it is finite.

Hence, the map $\operatorname{Prod}$ is representable, injective, quasi-finite and maps closed points to closed points. Furthermore, $\overline{\Product}$ is finite, and hence by \cite[Proposition 6.4]{Alper_good}, the map $\operatorname{Prod}$ is finite. Furthermore, it is an open immersion, and by Zariski's main theorem for stacks it is a closed immersion, i.e. it is an isomorphism. The map $\overline{\Product}$ is a birational morphism with finite fibers to a normal variety, $\overline{\Product}$ is an isomorphism to an open subset by Zariski's main theorem. But, $\overline{\Product}$ is also an open immersion, hence it is an isomorphism.
\end{proof}

Note that by the proof of Theorem \ref{main_thm} we see that if $X_1\not\cong X_2$ the map $\Product$ is not injective. 
% , as in particular, it is not representable, since the map on stabilizers $f\colon (\Aut(Y), \Aut(Z)) \rightarrow \Aut(Y\times Z)$ is not injective. 
We will, however, show that it is {\'e}tale and that in particular it is a $S_2$ gerbe.

\begin{theorem}\label{map is etale}
    The map $\Product$ is {\'e}tale. In particular, if $X_1\cong X_2$ is simple, it is a $S_2$-gerbe.
\end{theorem}
\begin{proof}
    Since by Theorem \ref{main_thm} the map is an isomorphism when $X_1\not\cong X_2$ have no common simple components in their product decompositions, we are left to show the statement when $X_1\cong X_2$ is simple. To achieve this, we will show that the map $\Product$ is a gerbe. We will achieve this by proving a stronger statement, mainly that $\M_X^K\cong \left[\M^K_{X_1}\times\M^K_{X_1}/S_2 \right]$, where $S_2$ is the symmetric group of order $2$. 

    Consider the quotient stack $\Y\coloneqq \left[\M^K_{X_1}\times\M^K_{X_1}/S_2 \right]$, and let $Y\coloneqq M^K_{X_1}\times M^K_{X_1}/S_2$. We then have a commutative diagram 
    \begin{center}
        \begin{tikzcd}
        \M^K_{X_1}\times\M^K_{X_1} \arrow[r, "q"]\arrow[d, "\pi_{X_1}\times \pi_{X_1}"] &\Y\arrow[d, "\pi_{\Y}"]\\
         M^K_{X_1}\times M^K_{X_1}\arrow[r, "\overline{q}"] & Y.
        \end{tikzcd}
    \end{center}
    Notice, that $\overline{q}\circ (\pi_{X_1}\times \pi_{X_1})$ is quasicompact and cohomologically affine, as the composition is exact, and hence the corresponding functor on $\mathrm{QCoh}(-)$ is exact. Hence, $\pi_{\Y}\circ q$ is cohomologically affine, and thus, the map $\pi_{\Y}$ is cohomologically affine, since cohomological affineness is preserved under compositions of morphisms. Furthermore, the morphism of quasi-coherent $\mathcal{O}_{\Y}$-modules $\pi_{\Y}^{\#}\colon \mathcal{O}_Y\rightarrow \pi_{\Y}^*\mathcal{O}_{\Y}$ pulls back under the morphism $\overline{q}$ to an isomorphism (since $\pi_{X_1}\times \pi_{X_1}$ is a good moduli space), so by descent, $\pi_{\Y}^{\#}$ is an isomorphism. Hence, $\Y\xrightarrow{\pi_{\Y}}Y$ is a good moduli space. Notice that $Y$ is projective, since the CM line bundle on $Y$ is ample as it pullbacks to the ample CM line bundle on $M^K_{X_1}\times M^K_{X_1}$, since $S_2$ is finite, which is ample by \cite{codogni,xu-zhuang}. Furthermore, the scheme $Y$ is proper by the same arguments as in \cite{Blum_Xu_2019, liu2021finite} (partly because it is separated and is a quotient by a finite group action). In particular, the stack $\Y$ is separated. 

    There exists a natural map $\Product_{\Y}\colon \Y \rightarrow \M^K_X$, which on the level of pseudofunctors sends $\Y \ni[A]\times [B] = [B]\times [A]\longmapsto [A\times B] \in \M^K_X$. We will show that the map $\Product_{\Y}$ is an isomorphism, which in turn will imply that the map $\Product$ is a gerbe. As in the proof of Theorem \ref{main_thm}, the map sends closed points to closed points, is separated, quasifinite, and the corresponding map $\overline{\Product_{\Y}}\colon Y\rightarrow M^K_X$, is proper and quasi-finite, and hence finite. In addition, by Theorem \ref{optimal result} the map $\Product_{\Y}$ is open. We are left to show that $\Product_{\Y}$ is representable, which from \cite[Proposition 6.4]{Alper_good} will imply that $\Product_{\Y}$ is an isomorphism. But, the map on stabilisers $(\Aut(A), \Aut(B))\rightarrow\Aut(A\times B) $ is injective by construction, hence $\Product_{\Y}$ is an isomorphism. In particular, this shows that $\Product$ is a $S_2$-gerbe.

    % Recall that 
    %  $$\M^K_{X_1}  = [Z^{\mathrm{red}}_1/ \PGL(N_1+1)]$$
    %  and
    %   $$\M^K_{X}  = [Z^{\mathrm{red}}/ \PGL(N+1)].$$
    % Hence, by Lemma \ref{products of quotient stacks} 
    % $$\M^K_{X_1}\times \M^K_{X_1} \cong [(Z^{\mathrm{red}}_1\times Z^{\mathrm{red}}_1)/ (\PGL(N_1+1)\times \PGL(N_1+1))].$$

    % Consider the following commutative diagram 
    % \begin{center}
    %     \begin{tikzcd}
    %     \left[Z^{\mathrm{red}}/S_2 \right] \arrow[r]\arrow[d] &\left[\M^K_{X_1}\times \M^K_{X_1}/ S_2\right]\arrow[d, "g"]\arrow[r] &\M^K_{X_1}\times \M^K_{X_1}\arrow[d, "\Product"]\\
    %     Z^{\mathrm{red}} \arrow[r, "f"] & \M^K_{X}\arrow[r, "\sim"]&  \M^K_{X}.
    %     \end{tikzcd}
    % \end{center}
    % Notice that both right and left square diagrams are Cartesian, and that 
    % $$\left[\M^K_{X_1}\times \M^K_{X_1}/ S_2\right]\times_{\M^K_{X}} Z^{\mathrm{red}}\cong  \left[Z^{\mathrm{red}}/S_2 \right],$$
    % and that the map $f$ is surjective, flat, and locally of finite presentation, as it is a principle $\PGL(N+1)$-bundle. Then, by \cite[Lemma 101.28.7]{stacks-project} the map $g$ is a gerbe. In particular, by \cite[Lemma 101.28.5]{stacks-project} the map $\Product$ is a gerbe.

    Since $\Product$ is a $S_2$-gerbe, it is surjective, flat, locally of finite presentation and smooth \cite[Proposition 101.28.11, Lemma 101.33.8]{stacks-project}. In particular, the diagonal map $\Delta_{\Product}$ is surjective, flat and of finite presentation, as all of these properties are preserved under base change. By the proof of Theorem \ref{main_thm} the map $\Product$ is separated, hence $\Delta_{\Product}$ is a closed immersion, which implies that it is locally of finite type, hence the diagonal map $\Delta_{\Product}$ is {\'e}tale. Since $\Product$ is a morphism of stacks which is locally of finite presentation, flat, and its diagonal $\Delta_{\Product}$ is {\'e}tale, the map $\Product$ is {\'e}tale by \cite[Lemma 101.36.10]{stacks-project}.

    To conclude the proof, let $X_1= Y\times Z$ and $X_2 = Z\times W$, where $Y$, $Z$, $W$ are all simple $\Q$-Fano varieties, and let $X = X_1\times X_2$. We will now show that $$\Product\colon \mathcal{M}^K_{X_1}\times \mathcal{M}^K_{X_2}\rightarrow \mathcal{M}^K_{X}$$
    is {\'e}tale. By Theorem \ref{main_thm} $\mathcal{M}^K_{X_1}\cong \mathcal{M}^K_{Y}\times \mathcal{M}^K_{Z}$ and $\mathcal{M}^K_{X_2}\cong \mathcal{M}^K_{Z}\times \mathcal{M}^K_{W}$. Consider the quotient stack $\Y\coloneqq \left[\M^K_{Z}\times\M^K_{Z}/S_2 \right]$, and let $\X\coloneqq \M^K_{Y}\times\M^K_{W}$, then there is a natural map 
    $$\Product_{\X,\Y}\colon \X\times \Y \rightarrow \M^K_X$$
    which on the level of pseudofunctors sends $\X\times\Y \ni([A\times B],[C]\times [D]) = ([A\times B], [D]\times [C])\longmapsto [A\times B\times C\times D] \in \M^K_X$. The map $\Product_{\X,\Y}$ sends closed points to closed points, is separated, quasifinite, and the corresponding map on good moduli spaces is proper and quasi-finite, and hence finite, by the same arguments as above. In addition, by Theorem \ref{optimal result} the map $\Product_{\X,\Y}$ is open. We are left to show that $\Product_{\Y}$ is representable, which follows by exactly the same argument as in the proof above. This shows that $\M^K_X \cong \X\times \Y$, which in particular shows that $\Product$ is a product of an {\'e}tale morphism and an $S_2$-gerbe. By the same argument as above we conclude that the map $\Product$ is {\'e}tale. 
\end{proof}
In fact, we have also shown the following.

\begin{corollary}\label{quot stack iso}
    We have $\M^K_{X\times X}\cong [(\M^K_X)^2/S_2]$, when $X$ is simple.
\end{corollary}
Let now 
\begin{equation*}
    X = \prod_{i=1}^r \Bigg(\prod_{j=1}^{s_i}X_{i,j}\Bigg).
\end{equation*}

\begin{corollary}\label{non-distinct irred components}
    Suppose that $\mathcal{M}^K_{X_{i,j}}\cong \mathcal{M}^K_{X_{i,j}}$ only if $i = i'$. 
    Then 
    \begin{equation*}
        \mathcal{M}^K_{X}\cong \prod_{i=1}^r \left[\big(\mathcal{M}_{X_{i,1}}^{K}\big)^{\times s_i} / S_{s_i}\right],
    \end{equation*}
    where $S_{s_i}$ is the symmetric group of order $s_i$, which  acts on $\big(\mathcal{M}_{X_{i,1}}^{K}\big)^{\times s_i}$ by permuting the factors. The quotient here is a quotient stack.
\end{corollary}

% \begin{remark}\label{rmk: arbitrary etale products}
%     The above corollary is extended to arbitrary products of the same copies of K-moduli stacks, where
%     $ \M^K_{\prod_{i=1}^m X} \cong \left[(\M^K_X)^{\times m}/S_m \right]$, and $S_m$ is the symmetric group of order $m$.
% \end{remark}

% Recall that the K-moduli stack $\M^K_{n,V}$ is defined as follows: 

% \[\mathcal{M}^{K}_{n,V}(S)\vcentcolon= \left\{ 
%     \begin{aligned}
%          \text{ flat proper morphisms } f\colon \mathcal{X}\rightarrow S \text{, with fibres  }\mathcal{X}_t  \text{ that are} \\
%         n\text{-dimensional }K\text{-semistable } \mathbb{Q}\text{-Fano varieties with volume }
%         % satisfying Koll{\'a}r's condition  (see, {\cite[24]{kollar2009hulls}})
%         V
%     \end{aligned}
% \right\}.\]

Furthermore, Theorem \ref{main_thm} can be generalised to Artin stacks $\X, \Y$, admitting proper good moduli spaces $X$, $Y$, with some conditions we will specify now. Let $[x]\in \X$, $[y]\in \Y$, and let $f\colon \X\rightarrow \Y$ be the map that sends $[x]\mapsto[y])$, such that $f$ is separated, an open immersion, and maps closed points to closed points. 
% The map $\Product$ is defined in a way that generalises definition \ref{product map def}, for arbitrary stacks.

% $[x\times y]\coloneqq \mathrm{Im}([x],[y]) \in \mathcal{Z}$, and the map $\X\times \Y\rightarrow\mathcal{Z}$, which is defined by sending $([x],[y])\mapsto [x\times y]$, is injective on the map of stabilisers. 
% We provide a brief proof of this fact below.

\begin{proposition}\label{prop: gms of products of artin stacks}
    Let $\X$, $\Y$ be Artin stacks admitting proper good moduli spaces $X$, $Y$. Assume that $f\colon \X\rightarrow \Y$ is representable, and that it maps closed points to closed points. Then $f$ is quasi-finite. If $f$ is also an open immersion, then $f$ is an isomorphism, and descends to an isomorphism of good moduli spaces 
    $$\overline{f}\colon X\xrightarrow{\sim} Y.$$
\end{proposition}
\begin{proof}
    
    % In addition, $f$ is representable by assumption and injective by construction, hence it is an open immersion.

    We will first show that $\Product$ is quasi-finite. Let $\pi_X\colon \X\rightarrow X$ be the good moduli space map. Recall by \cite[Theorem 4.12]{Alper_Hall_Rydh_2020} that for each $x\in |X|$ and $x_0\in |\X|$, which is the unique-closed point of $\pi^{-1}_{X}(x)$. Then there exists a ring $A_X$ such that we have a diagram 
    \begin{center}
    \begin{tikzcd}
    (\mathcal{W_X},w_X) = [\operatorname{Spec}A_X /G_{x} ]\arrow[r, "f_X "]\arrow[d] & (\X,x)\arrow[d, "\pi_X"]\\
    W_X = \operatorname{Spec} A_X^{G_x} \arrow[r, "g_X"]& X,
    \end{tikzcd}
\end{center}
    where $W_X$ is the good moduli space of $\W_X$, and the maps $f_X$ and $g_X$ are {\'e}tale. Similar diagrams exist for $\Y$. Let $P_X$, $P_Y$ be a $G_x$-torsor and a $G_y$-torsor respectively,
    % , and let $P = P_X\times P_Y$. 
    Since $f$ maps closed points to closed points we obtain the following diagram:
    \begin{center}
        \begin{tikzcd}
        P_X \arrow[r]\arrow[d]\ar[rr,bend left=35] &\left[P/G_x\right]\arrow[d]\arrow[r, "h"] &\mathrm{Spec}A_Y\arrow[d]\\
        \mathrm{Spec}A_X \arrow[r]\arrow[d, "="] & \mathcal{W}_X\times \mathcal{W}_Y \arrow[r, "f_{W}"]\arrow[d, "g_X"]&  \mathcal{W}_Y\arrow[d, "g_{Y}"]\\
        \mathrm{Spec}A_X\arrow[r] & \X\arrow[r, "f"] & \Y,
        \end{tikzcd}
    \end{center}
    % \begin{center}
    %     \begin{tikzcd}
    %     P \arrow[r]\arrow[d]\ar[rr,bend left=35] &\left[P/\big(G_x\times G_y\big)\right]\arrow[d]\arrow[r, "h"] &\mathrm{Spec}A_Z\arrow[d]\\
    %     \mathrm{Spec}A_X\times \mathrm{Spec}A_Y \arrow[r]\arrow[d, "="] & \mathcal{W}_X\times \mathcal{W}_Y \arrow[r, "\Product_{W}"]\arrow[d, "f"]&  \mathcal{W}_Z\arrow[d, "f_{Z}"]\\
    %     \mathrm{Spec}A_X\times \mathrm{Spec}A_Y\arrow[r] & \X\times \Y\arrow[r, "\Product"] & \mathcal{Z},
    %     \end{tikzcd}
    % \end{center}
    where $h$ is $G_x$-equivariant. Since quasi-finiteness can be checked {\'e}tale locally, we will show $h$ is quasi-finite, which will in turn show that $\Product_W$ and hence $\Product$ are quasi-finite. Showing that $f_W$ is quasi-finite follows from the same argument as in the proof of Theorem \ref{main_thm}. Hence, $f$ and $\overline{f}$ are both quasi-finite. In particular, $\overline{f}$ is finite since it is proper. 
    
    If $f$ is an open immersion, as in the proof of Theorem \ref{main_thm} we will use \cite[Proposition 6.4]{Alper_good} to show that $\Product$ is finite, and then the claim will follow from Zariski's main theorem. 
    % The proof for separatedness follows from the proof of Theorem \ref{main_thm}, noting that $X$, $Y$ are separated, and $f$ maps closed points to closed points. 
    Since $f$ is an open immersion, it is separated. 
    Hence, by \cite[Proposition 6.4]{Alper_good} the map $f$ is finite, and as in the proof of Theorem \ref{main_thm}, both $f$ and $\overline{f}$ are isomorphisms.
\end{proof}

% We also extend Theorem \ref{map is etale} as follows. 

% \begin{proposition}\label{s2-gerbe for stacks}
%     Let $\X$, $\mathcal{Z}$ be Artin stacks admitting proper good moduli spaces $X$, $Z$, such that for $[x], [y]\in \X$, $[z] = \Product([x],[y])\in \mathcal{Z}$. Assume that the map $\Product\colon \X\times \X \rightarrow \cZ$ is separated, open and that it maps closed points to closed points. Then $\Product$ is {\'e}tale.
% \end{proposition}
% \begin{proof}
%     The proof follows from the proof of Theorem \ref{map is etale}, noting that the map $\Product$ is a separated gerbe in this case. We see this by taking the affine GIT quotient stack $(\mathcal{W_X},w_X) = [\operatorname{Spec}A_X /G_{x} ]$ with corresponding moduli space $W_X$, over a closed point $x\in |X|$, as in the proof of Proposition \ref{prop: gms of products of artin stacks}, and following the strategy of the proof of Theorem \ref{map is etale}.
% \end{proof}

We can also extend Theorem \ref{main_thm} to the case of log Fano pairs. As before, let $(X_i,c_iD_i)$ be log Fano pairs, for $0<c_i<r_i^{-1}$, with $D_i\sim_{\Q}-K_{X_i}$. Let also $X = X_1\times X_2$, and let $\mathbf{c}D\coloneq c_1D_1\boxtimes c_2D_2 = c_1p_1^*D_1+c_2p_2^*D_2$, where $p_i\colon X\rightarrow X_i$ are the projections. The pair $(X,\mathbf{c}D)$ is a log Fano pair for all $0<c_i<r_i^{-1}$. Let ${\M}^K_{X_1\times X_2, D, \mathbf{c}}$ be K-moduli stack parametrising K-semistable log Fano pairs $(X,\mathbf{c}D)$ up to S-equivalence.

\begin{theorem}\label{product thm log pairs}
    Suppose that $(X_1,c_1D_1)\not\cong (X_2,c_2D_2)$, where the $X_i$ have no common irreducible components. Then Then the map $\Product_{c_1,c_2}$ is an isomorphism of stacks which descends to an isomorphism of good moduli spaces. In particular, 
    $$ {\M}^K_{X_1, D_1, c_1}\times {\M}^K_{X_2, D_2, c_2}\cong {\M}^K_{X_1\times X_2, D, \mathbf{c}}$$
    and
    $${M}^K_{X_1, D_1, c_1}\times {M}^K_{X_2, D_2, c_2}\cong {M}^K_{X_1\times X_2, D, \mathbf{c}}.$$
\end{theorem}
\begin{proof}
    The fact that the map $\Product_{c_1,c_2}$ maps closed points to closed points, is separated, quasi-finite and that $\overline{\Product_{c_1,c_2}}$ is finite follow from the same arguments as the proof of Theorem \ref{main_thm}.

    We are left to check representability. The map on stabilisers is given by 
    $$f\colon (\Aut(Y,D_1), (\Aut(Z,D_2)) \rightarrow \Aut(Y\times Z, D_1\boxtimes D_2)$$
    where $[(Y,D_1)]$ is a closed point in $\M^K_{X_1,D_1,c_1}$ and $[(Z,D_2)]$ is a closed point in $\M^K_{X_2,D_2,c_2}$. Since we assume that $(X_1,c_1D_1)\not\cong (X_2,c_2D_2)$, the map $f$ is injective, and hence $\Product_{c_1,c_2}$ is representable. In fact, by a similar argument to the proof of Theorem \ref{main_thm}, $\Product_{c_1,c_2}$ is injective. Hence, by \cite[Proposition 6.4]{Alper_good} and Zariski's main theorem, it is an isomorphism, as intended. The map $\overline{\Product}_{c_1,c_2}$ is also an isomorphism, by the uniqueness of good moduli spaces.
\end{proof}
% Hence, it is not an isomorphism. We are, however, able to show the following.

\begin{remark}\label{etalness of prod map}
    Using similar methods as above, we can show that when ${\M}^K_{X_1, D_1, c_1}\cong {\M}^K_{X_2, D_2, c_2}$ the product map is {\'e}tale and an $S_2$-gerbe, but we omit the proof.
\end{remark}

\section{Explicit Examples}\label{sec: examples}

\subsection{K-moduli of Fano threefolds of Picard rank greater than 6}\label{sec:Fano threefolds moduli}

Let $X = \pr^1$ and let $Y_{9-n} = \mathrm{Bl}_{n}\pr^2$ be the blow up of $\pr^2$ in $8\geq n\geq 5$ points in general position, i.e. the del Pezzo surface of degree 1, 2, 3, 4. The Fano threefold $X\times Y_{9-n}$ is a smooth member of families \textnumero 7.1, 8.1, 9.1 and 10.1 when $n = 5$, $6$, $7$, $8$ respectively. Smooth Fano threefolds $X\times Y_{9-n}$ are known to be K-polystable by \cite{Tian_Yau_1987, Tian_1990} and Theorem \ref{zhuang_thm}. In addition, since by \cite{mabuchi_mukai_1990, odaka_spotti_sun_2016} there exists a complete description of singular K-polystable and K-semistable del Pezzo varieties of the above degrees, there is also a comprehensive description of the singular K-polystable and K-semistable Fano threefolds in the above families. The only missing piece in the full description of the moduli spaces of the above families of Fano threefolds, are possible K-polystable singular degenerations that do not appear as products $X\times Y_{9-n}$. In this section, we will use Theorem \ref{main_thm} to obtain a full description of these K-moduli.

Let $\M^K_d$ denote the K-moduli stack of K-semistable del Pezzo varieties of degree $1\leq d\leq 4$, that admits a good moduli space $M^K_d$, parametrising K-polystable del Pezzo varieties of degree d. Let $\M^K_X$ be the K-moduli stack of Fano variety $\pr^1$, where $\M^K_X\cong M^K_X\cong \{\mathrm{pt}\}$. Finally, let $\M^K_{X\times Y_d}$ be the K-moduli stack of K-semistable Fano threefolds $X\times Y_d$ of family \textnumero 7.1, 8.1, 9.1 and 10.1, where $d$ is the corresponding degree of the del Pezzo surface $Y_d$, with good moduli space $M^K_{X\times Y_d}$.

\begin{proposition}\label{fano_threefolds_ex}
    We have $\M^K_{X\times Y_d}\cong \M^K_d$. Furthermore, at the level of good moduli spaces we have 
    \begin{enumerate}
        \item $M^K_{X\times Y_4}\cong \mathrm{Gr}\big(2, \mathrm{Sym}^2(\C^5)\big)\sslash \PGL(5)\cong \pr(1,2,3)$;
        \item $M^K_{X\times Y_3}\cong \pr\big(H^0(\pr^3, \cO_{\pr^3}(3)\big)\sslash \PGL(4)\cong \pr(1,2,3,4,5)$;
        \item $M^K_{X\times Y_2}\cong \operatorname{Bl}_{[2C]}\pr\big(H^0(\pr^2, \cO_{\pr^2}(4)\big)\sslash \PGL(3)$,
        where $2C$ is the double conic in $\pr^2$.
    \end{enumerate}
\end{proposition}
\begin{proof}
    The proof follows directly from Theorem \ref{main_thm} and \cite{mabuchi_mukai_1990, odaka_spotti_sun_2016}.
\end{proof}

\begin{remark}
    If, instead of $\pr^1$, we considered $X= \pr^n$, and the $n+1$-dimensional Fano variety $X\times Y_d$, the above result would be extended accordingly, since $\M^K_X\cong M^K_X\cong \{\mathrm{pt}\}$. In fact, this is true for any choice of Fano variety $Y\not\cong \pr^n$.
\end{remark}

\begin{remark}
    The result of Theorem \ref{main_thm} which is applied to explicitly describe the K-moduli spaces of Fano threefolds of high Picard rank as in Proposition \ref{fano_threefolds_ex}, can also be used to provide explicit descriptions of K-moduli of various Fano varieties. We choose to present only the above here, since the study of explicit K-moduli of Fano threefolds is a popular topic in recent research \cite{liu-xu, liu2020kstability, spotti_sun_2017, pap22, abban2023onedimensional}.
\end{remark}

\subsection{Polyhedral K-moduli wall-crossings}\label{sec: pairs examples}
In \cite{zhou2023shape} the author showed that when log Fano pairs $(X, \sum_{i=1}^kc_iD_i)$, where $D_i\cong_{\Q}-K_X$ and $c_i\in \Q$, there exists a finite wall-chamber decomposition of the K-moduli stack $\M^K_{X, \sum_{i=1}^kc_iD_i}$ (c.f. \cite[Theorem 1.6]{zhou2023shape}). We will provide an explicit example of such a wall-crossing.

\subsubsection{Prior results on K-moduli wall-crossings}\label{sec: prior results}

In \cite{ascher2019wall, zhou_wall-crossing}, it was shown that the K-moduli stack $\M^K_{X,cD}$, of log Fano pairs, with $D\sim_{\Q} -K_X$, undergoes a wall-chamber decomposition. Specific instances of this phenomenon were described explicitly in \cite{ascher2019wall, ascher2020kmoduli,Gallardo_2020, pap22, zhao2023moduli,me_junyan_jesus} for the cases when $X$ is a del Pezzo surface, and $D\sim -K_X$. We summarise some of the results here, which will be useful in Section \ref{sec: pairs examples}.

Let $X_1\subset \pr^3$ be a del Pezzo surface of degree $3$, i.e. a cubic surface, and let $X_2\subset \pr^4$ be a del Pezzo surface of degree $4$, i.e. a complete intersection of two quadrics in $\pr^4$. Let $H_1 \cong \pr^2\subset \pr^3$ and $H_2 \cong \pr^3\subset \pr^4$ be two hyperplanes in $\pr^3$ and $\pr^4$, respectively, where $D_i \coloneqq X_i\cap H_i$ are hyperplane sections such that $D_i \sim -K_{X_i}$. The pairs $(X_i,c_iD_i)$ are log Fano pairs for $0\leq c_i<1$, with K-moduli stacks $\M^K_{X_i,D_i,c_i}$ which have been studied extensively \cite{Gallardo_2020, pap22, me_junyan_jesus}. In particular, it was shown that $\M^K_{X_1,D_1,c_1} \cong \M^{\GIT}_3(t(c_1))$ and $\M^K_{X_2,D_2,c_2} \cong \M^{\GIT}_4(t(c_2))$, where 
$$\M^{\GIT}_3(t(c_1)) = [(\pr H^0(\pr^3,\O_{\pr^3}(3))\times H^0(\pr^3,\O_{\pr^3}(1)))^{ss}/ \PGL(4)]$$
is the GIT moduli stack parametrising pairs $(X_1,H_1)$, of cubics and hyperplanes in $\pr^3$, with respect to a polarisation $\L_{X_1, t(c_1)}$, and similarly 
$$\M^{\GIT}_4(t(c_2)) = [(\pr \operatorname{Gr}(2, H^0(\pr^4,\O_{\pr^4}(2)))\times H^0(\pr^4,\O_{\pr^4}(1)))^{ss}/ \PGL(4)]$$
is the GIT moduli stack parametrising pairs $(X_2,H_2)$, of complete intersections of two quadrics and hyperplanes in $\pr^4$, with respect to a polarisation $\L_{X_2, t(c_2)}$. These GIT moduli stack admit projective good moduli spaces, which are the corresponding GIT quotients in each case.

We have the following prior results that give a full description of the walls and chambers in $\M^{\GIT}_3(t(c))$ and $\M^{\GIT}_4(t(c))$, and consequently in $\M^K_{X_1,D_1,c_1}$ and $\M^K_{X_2,D_2,c_2}$.

\begin{theorem}\textup{(cf. \cite[Lemma 1]{Gallardo_2018} and \cite[Theorem 3.8]{Gallardo_2020})}\label{jesus+patricio}
The values of the walls $t$ in $(0,1)$ for ${M}^{\GIT}_3(t)$ are:
$$t_0=0,\quad t_1=\frac{1}{5},\quad t_2=\frac{1}{3},\quad t_3=\frac{3}{7},\quad t_4=\frac{5}{9},\quad t_5=\frac{9}{13},\quad t_6=1.$$
In particular, $$t(c)=\frac{9c}{8+c}.$$
\end{theorem}

The walls of $\M^K_{X_1,D_1,c_1}$ in terms of $c_1$ are 
$$c_{1,0} = 0,\quad c_{1,1} = \frac{2}{11}, \quad c_{1,2} = \frac{4}{13},\quad c_{1,3} = \frac{2}{5},\quad c_{1,4} = \frac{10}{19},\quad c_{1,5} = \frac{2}{3}, c_{1,6} = \quad 1.$$

\begin{theorem}\textup{(cf. \cite[\S 6.3, Theorem 6.20, Lemma 7.5]{pap22})}\label{my_work}
The values of the walls $t$ in $(0,1)$ for ${M}^{\GIT}_4(t)$ are $$t_0=0,\quad t_1=\frac{1}{6},\quad t_2=\frac{2}{7},\quad t_3=\frac{3}{8},\quad t_4=\frac{6}{11},\quad t_5=\frac{2}{3},\quad t_6=1.$$
In particular, $$t(c)=\frac{6c}{5+c}.$$
\end{theorem}

The walls of $\M^K_{X_2,D_2,c_2}$ in terms of $c_2$ are 
$$c_{2,0} = 0,\quad c_{2,1} =\frac{1}{7},\quad c_{2,2} =\frac{1}{4},\quad c_{2,3} =\frac{1}{3},\quad c_{2,4} =\frac{1}{2},\quad c_{2,5} =\frac{5}{8}, \quad c_{2,6} =1.$$

\begin{theorem}[{\cite[Theorem 1.1]{me_junyan_jesus}}]\label{thm: junyan-me-jesus}
    Let $c\in(0,1)$ be a rational number and $d=3,4$. Let $t(c)=\frac{9c}{8+c}$ if $d=3$, and $t(c)=\frac{6c}{5+c}$ if $d=4$. There is an isomorphism between Artin stacks $${\M}^{K}_d(c)\simeq {\M}^{\GIT}_d(t(c)),$$ which descends to an isomorphism ${M}^{K}_d(c)\simeq {M}^{\GIT}_d(t(c))$ between the corresponding good moduli spaces. In particular, these isomorphisms commute with the wall-crossing morphisms.
\end{theorem}

\subsubsection{Wall crossings for Fano threefolds of high Picard rank}
Let $X_d = X\times Y_{d}$, where $4\leq d\leq 1$ is the degree of the del Pezzo surface $Y_d$, and consider now log Fano pairs $(X_d, cD)$, where $D\sim -K_{X_d}$. Here, $D = p_1^*(-K_{\pr^1}) - p_2^*(-K_{Y_{d}})$. Let $\M^K_{X_d,c}$ be the K-moduli stack parametrising K-semistable log Fano pairs $(X_d, cD)$. Let $D_1 \sim -K_{\pr^1}$ and $H_d\sim -K_{Y_d}$. As before, let $\M^K_{d, c}$ be the K-moduli stack parametrising K-semistable log Fano pairs $(Y_d, cH_d)$, and let $\M^K_{\pr^1,c}$ be the K-moduli stack parametrising K-semistable log Fano pairs $(\pr^1,D_1)$. Notice that for all $0\leq c<1$, the K-moduli stack $\M^K_{\pr^1,c}\cong \{\mathrm{pt}\}$, as the log Fano pair $(\pr^1,D_1)$ is K-semistable for all $c$, and $\pr^1$ doesn't admit any $\Q$-Gorenstein degenerations. We also denote by $\M^K_{\pr^2}(c)$ the moduli stack of log Fano pairs with a $\Q$-Gorenstein smoothing to pairs $(\mathbb P^2, \frac{1}{2}C+cC')$, where $C$ is a quartic curve and $C'$ is a line in $\pr^2$, and its good moduli space by ${M}^K_{\pr^2}(c)$. Let also 
$$\M^{\GIT}_{\mathbb P^2}(t)\coloneqq\left[\left(\mathbb P H^0\left(\mathbb P^2, \mathcal O_{\mathbb P^2}(4)\right)\times\mathbb PH^0\left(\mathbb P^2, \mathcal O_{\mathbb P^2}(1)\right)\right)^{ss}/_{\cO(a,b)} \PGL(3)\right],$$ 
be the GIT moduli stack parametrising GIT semistable pairs $(C,C')$ in $\pr^2$, with corresponding good moduli space $M^{\GIT}_{\mathbb P^2}(t)$. By \cite[Theorem 1.2]{me_junyan_jesus}, the stack $\M^K_{\pr^2}(c)$ is the quotient stack obtained by taking the Kirwan blow-up $\left[{\U}^{ss}_{t(c)}/\PGL(3)\right]$ of the VGIT moduli stack $\M^{\GIT}_{\pr^2}(t(c))$ along the locus parameterizing the pair $(2C,L)$, where $C$ is a smooth conic and $L$ is a line intersecting $C$ transversely.

We then obtain the following wall-crossing picture for K-moduli of log Fano pairs $(X_d, cD)$.

\begin{proposition}\label{fano_threefolds_wall_crossing}
    We have $\M^K_{X_d,c}\cong \M^K_{d,c}$. Furthermore, we have 
    \begin{enumerate}
        \item $\M^K_{X_4, c}\cong \M^{\GIT}_4(t(c))$;
        \item $\M^K_{X_3, c}\cong \M^{\GIT}_3(t(c))$;
        \item $M^K_{X_2,c} \cong {\U}^{ss}_t\sslash_{\O(a,b)}\PGL(3)$.
    \end{enumerate}
\end{proposition}
\begin{proof}
    The proof follows directly from Theorems \ref{jesus+patricio}, \ref{my_work} and \ref{product thm log pairs}, and \cite[Theorem 1.2]{me_junyan_jesus}.
\end{proof}

% \begin{remark}
%     Proposition \ref{fano_threefolds_wall_crossing} gives an explicit description for all the K-moduli wall crossings, and the wall-chamber decompositions for all log Fano pairs $(X, cD)$, where $D\sim -K_X$ and $X$ is a Fano threefold in families \textnumero 7.1, 8.1, 9.1.
% \end{remark}

% \subsection{K-moduli wall crossings for Fano threefolds}

\subsubsection{Polyhedral wall crossings for higher dimensional Fano varieties}

Let $X_1\subset \pr^3$ be a del Pezzo surface of degree $3$, and  $X_2\subset \pr^4$ be a del Pezzo surface of degree $4$. Let $H_1 \cong \pr^2\subset \pr^3$ and $H_2 \cong \pr^3\subset \pr^4$ be two hyperplanes in $\pr^3$ and $\pr^4$, respectively, where $D_i \coloneqq X_i\cap H_i$ are hyperplane sections such that $D_i \sim -K_{X_i}$. The pairs $(X_i,c_iD_i)$ are log Fano pairs for $0\leq c_i<1$, with K-moduli stacks $\M^K_{X_i,D_i,c_i}$ which have been explicitly described in Theorems \ref{jesus+patricio} and \ref{my_work} for all walls and chambers. As before, these have been described in terms of the VGIT quotient stacks $\M^{\GIT}_3(t(c_1))$ and $\M^{\GIT}_4(t(c_2))$, where 
 $$\M^{\GIT}_3(t(c_1)) = [(\pr H^0(\pr^3,\O_{\pr^3}(3))\times \pr H^0(\pr^3,\O_{\pr^3}(1)))^{ss}/ \PGL(4)]$$
is the GIT moduli stack parametrising pairs $(X_1,H_1)$, of cubics and hyperplanes in $\pr^3$, with respect to a polarisation $\L_{X_1, t(c_1)}$, and similarly 
$$\M^{\GIT}_4(t(c_2)) = [( \operatorname{Gr}(2, H^0(\pr^4,\O_{\pr^4}(2)))\times \pr H^0(\pr^4,\O_{\pr^4}(1)))^{ss}/ \PGL(5)]$$
is the GIT moduli stack parametrising pairs $(X_2,H_2)$, of complete intersections of two quadrics and hyperplanes in $\pr^4$, with respect to a polarisation $\L_{X_2, t(c_2)}$.

Let $X = X_1\times X_2$ and $\mathbf{c}D = c_1D_1\boxtimes c_2D_2 = c_1p_1^*D_1+c_2p_2^*D_2$, where $p_i\colon X\rightarrow X_i$ are the two corresponding projections. The pair $(X, \mathbf{c}D)$ is a log Fano pair for $0<c_1$, $c_2<1$, with $p_1^*D_1+p_2^*D_2\sim -K_X$.
% Combining these results with Theorem \ref{main_thm}, we obtain the following result. 
Let $\M^K_{X, \mathbf{c}D}$ be the K-moduli stack parametrising K-semistable log Fano pairs $(X, \mathbf{c}D)$, with good moduli space $M^K_{X, \mathbf{c}D}$ as in Sections \ref{sec:prelims} and \ref{sec:product map}. We will combine the results of Theorems \ref{jesus+patricio}, \ref{my_work} and \ref{thm: junyan-me-jesus} and \cite{me_junyan_jesus}, with Theorem \ref{product thm log pairs} to obtain a full characterisation of the moduli stack $\M^K_{X, \mathbf{c}D}$ and its polyhedral wall-crossings, in line with \cite{zhou2023shape}.

\begin{theorem}\label{product_example for pairs}
For all $c_1$, $c_2$, $\mathbf{c} = (c_1,c_2)$ we have
    $$\M^K_{X, \mathbf{c}D}\cong \M^K_{X_1,D_1,c_1}\times \M^K_{X_2,D_2,c_2}\cong \M^{\GIT}_3(t(c_1))\times \M^{\GIT}_4(t(c_2))$$
    and
    $$M^K_{X, \mathbf{c}D}\cong M^K_{X_1,D_1,c_1}\times M^K_{X_2,D_2,c_2}\cong M^{\GIT}_3(t(c_1))\times M^{\GIT}_4(t(c_2)).$$
    In particular, there are 36 total wall chamber decompositions, as illustrated in Figure \ref{fig:wall-chamber}.
\end{theorem}

\begin{figure}[h!]
    \centering
    \begin{tikzpicture}
\begin{axis}[
    xlabel={$c_1$},
    ylabel={$c_2$},
    xmin=0, xmax=1.2,
    ymin=0, ymax=1.2,
    xtick={0,0,0,0,0,0,0,2/11,4/13,2/5,10/19,2/3,1},
    ytick={0,1/7,1/4,1/3,1/2,5/8,1},
    axis lines=middle,
    width=10cm,
    height=10cm
    % xticklabel style={
    %     /pgf/number format/frac,
    %     /pgf/number format/frac whole=false,
    %     /pgf/number format/frac denom=11
    % },
    % yticklabel style={
    %     /pgf/number format/frac,
    %     /pgf/number format/frac whole=false,
    %     /pgf/number format/frac denom=7
    % }
]
% Points
\addplot[only marks, mark=*] coordinates {
    (0,0) (0,1/7) (0,1/4) (0,1/3) (0,1/2) (0,5/8) (0,1) 
    (2/11,0) (4/13,0) (2/5,0) (10/19,0) (2/3,0) (1,0)
};
% Vertical lines

\addplot[color=black] coordinates {
(0,1/7) (1,1/7)
};

\addplot[color=black] coordinates {
 (0,1/4) (1,1/4)
};
\addplot[color=black] coordinates {
(0,1/3) (1,1/3)
};
\addplot[color=black] coordinates {
(0,1/2) (1,1/2)
};
\addplot[color=black] coordinates {
(0,5/8) (1,5/8)
};
\addplot[color=black] coordinates { 
    (0,1)   (1,1)
};
% Horizontal lines
% \addplot[dashed] coordinates {
%     (2/11,0) (2/11,1)
%     (4/13,0) (4/13,1)
%     (2/5,0) (2/5,1)
%     (10/19,0) (10/19,1)
%     (2/3,0) (2/3,1)
%     (1,0)   (1,1)
% };
\addplot[color=black] coordinates {
    (2/11,0) (2/11,1)
};
\addplot[color=black] coordinates {
    (4/13,0) (4/13,1)
};
\addplot[color=black] coordinates {
    (2/5,0) (2/5,1)
};
\addplot[color=black] coordinates {
    (10/19,0) (10/19,1)
};
\addplot[color=black] coordinates {
    (2/3,0) (2/3,1)
};
\addplot[color=black] coordinates {
    (1,0)   (1,1)
};
\end{axis}
\end{tikzpicture}
    \caption{Polyhedral structures of wall-chamber decomposition and wall-crossing for $\M^K_{X,\mathbf{c}D}$.}
    \label{fig:wall-chamber}
\end{figure}
%possible examples: products of dP_i,D with dP_j,D. Products of P^n,cD with some other.

\begin{remark}
    The above example is a representative example of a plethora of possible descriptions of K-moduli wall-crossings one can obtain when relying on previous results, such as \cite{ascher2019wall, ascher2020kmoduli, pap22, me_junyan_jesus}. It should however be noted, that the examples generated using Theorem \ref{product thm log pairs} yield only ``rectangular'' wall-crossings, as in Figure \ref{fig:wall-chamber}.
\end{remark}

% \bibliographystyle{alpha}
% \bibliography{references}
\printbibliography

@article{Gallardo_2018,
	doi = {10.1090/proc/13950},
  
	url = {https://doi.org/10.1090%2Fproc%2F13950},
  
	year = 2018,
  
	publisher = {American Mathematical Society ({AMS})},
  
	volume = {146},
  
	number = {6},
  
	pages = {2395--2408},
  
	author = {Patricio Gallardo and Jesus Martinez-Garcia},
  
	title = {Variations of geometric invariant quotients for pairs, a computational approach},
  
	journal = {Proceedings of the American Mathematical Society}
}

@article{odaka_spotti_sun_2016, 
title={Compact moduli spaces of {D}el {P}ezzo surfaces and {K}\"{a}hler–{E}instein metrics}, 
volume={102}, 
DOI={10.4310/jdg/1452002879},
number={1}, 
journal={Journal of Differential Geometry}, 
author={Odaka, Yuji and Spotti, Cristiano and Sun, Song},
year={2016},
pages={127–172}
}

@article{spotti_sun_2017, 
title={Explicit {G}romov–{H}ausdorff compactifications of moduli spaces of {K}{\"a}hler–{E}instein {F}ano manifolds}, 
volume={13},
number={3}, 
journal={Pure and Applied Mathematics Quarterly}, 
author={Spotti, Cristiano and Sun, Song}, 
year={2017}, 
pages={477–515}}

@article {Li-Wang-Xu,
    AUTHOR = {Li, Chi and Wang, Xiaowei and Xu, Chenyang},
     TITLE = {On the proper moduli spaces of smoothable {K}\"{a}hler-{E}instein
              {F}ano varieties},
   JOURNAL = {Duke Math. J.},
  FJOURNAL = {Duke Mathematical Journal},
    VOLUME = {168},
      YEAR = {2019},
    NUMBER = {8},
     PAGES = {1387--1459},
      ISSN = {0012-7094},
   MRCLASS = {14J45 (14D20 14J10 53C25 53C55)},
  MRNUMBER = {3959862},
MRREVIEWER = {Indranil Biswas},
       DOI = {10.1215/00127094-2018-0069},
       URL = {https://0-doi-org.serlib0.essex.ac.uk/10.1215/00127094-2018-0069},
}

@article{Zhuang_2020, title={Product theorem for K-stability}, volume={371}, DOI={10.1016/j.aim.2020.107250}, journal={Advances in Mathematics}, author={Zhuang, Ziquan}, year={2020}, pages={107250}}

@article{Dervan_2016, title={On K-stability of finite covers}, volume={48}, DOI={10.1112/blms/bdw029}, number={4}, journal={Bulletin of the London Mathematical Society}, author={Dervan, Ruadhaí}, year={2016}, pages={717–728}}

@article{liu2020kstability, title={K-stability of cubic fourfolds}, volume={2022}, DOI={10.1515/crelle-2022-0002}, number={786}, journal={Journal für die reine und angewandte Mathematik (Crelles Journal)}, author={Liu, Yuchen}, year={2022}, pages={55–77}}

@article {liu-xu,
    AUTHOR = {Liu, Yuchen and Xu, Chenyang},
     TITLE = {K-stability of cubic threefolds},
   JOURNAL = {Duke Math. J.},
  FJOURNAL = {Duke Mathematical Journal},
    VOLUME = {168},
      YEAR = {2019},
    NUMBER = {11},
     PAGES = {2029--2073},
      ISSN = {0012-7094},
   MRCLASS = {14L24 (14E30 14J30 32Q20)},
  MRNUMBER = {3992032},
MRREVIEWER = {Yuji Odaka},
       DOI = {10.1215/00127094-2019-0006},
       URL = {https://0-doi-org.serlib0.essex.ac.uk/10.1215/00127094-2019-0006},
}

@misc{pap22,
  doi = {10.48550/ARXIV.2212.09332},
  
  url = {https://arxiv.org/abs/2212.09332},
  
  author = {Papazachariou, Theodoros Stylianos},
  
  title = {K-moduli of log Fano complete intersections},
  
  publisher = {arXiv},
  eprint={2212.09332},
  archivePrefix={arXiv},
  year = {2022},
  
  copyright = {arXiv.org perpetual, non-exclusive license}
}

@article{mabuchi_mukai_1990,
author = {Mabuchi, Toshiki and Mukai, Shigeru},
year = {1990},
volume = {145},
pages = {133-169},
place = {New York},
title = {Stability and {E}instein-{K}{\"a}hler metric of a quartic
del {P}ezzo surface, {E}instein metrics and {Y}ang-{M}ills connections (Sanda, 1990)},
journal = {Lecture Notes in Pure and Appl. Math.},
}

@article{Kovacs_2009, title={Young Person’s Guide to Moduli of higher dimensional varieties}, DOI={10.1090/pspum/080.2/2483953}, journal={Algebraic Geometry}, author={Kovács, Sándor J.}, year={2009}, pages={711–743}}

@article{Bhatt_Ho_Patakfalvi_Schnell_2013, title={Moduli of products of stable varieties}, volume={149}, DOI={10.1112/s0010437x13007288}, number={12}, journal={Compositio Mathematica}, author={Bhatt, Bhargav and Ho, Wei and Patakfalvi, Zsolt and Schnell, Christian}, year={2013}, pages={2036–2070}}

@misc{zhou2023shape,
      title={On the shape of K-semistable domain and wall crossing for K-stability}, 
      author={Chuyu Zhou},
      year={2023},
      eprint={2302.13503},
      archivePrefix={arXiv},
      primaryClass={math.AG}
}

@article {zhou_wall-crossing,
    AUTHOR = {Zhou, Chuyu},
     TITLE = {On wall-crossing for {K}-stability},
   JOURNAL = {Adv. Math.},
  FJOURNAL = {Advances in Mathematics},
    VOLUME = {413},
      YEAR = {2023},
     PAGES = {Paper No. 108857, 26},
      ISSN = {0001-8708,1090-2082},
   MRCLASS = {14J45 (14D23 14E30)},
  MRNUMBER = {4533746},
       DOI = {10.1016/j.aim.2022.108857},
       URL = {https://doi.org/10.1016/j.aim.2022.108857},
}

@misc{ascher2019wall,
      title={Wall crossing for {K}-moduli spaces of plane curves}, 
      author={Kenneth Ascher and Kristin DeVleming and Yuchen Liu},
      year={2019},
      eprint={1909.04576},
      archivePrefix={arXiv},
      primaryClass={math.AG}
}

@misc{abban2023onedimensional,
      title={One-dimensional components in the K-moduli of smooth Fano 3-folds}, 
      author={Hamid Abban and Ivan Cheltsov and Elena Denisova and Erroxe Etxabarri-Alberdi and Anne-Sophie Kaloghiros and Dongchen Jiao and Jesus Martinez-Garcia and Theodoros Papazachariou},
      year={2023},
      eprint={2309.12518},
      archivePrefix={arXiv},
      primaryClass={math.AG}
}

@article{ascher2020kmoduli, title={K-Moduli of curves on a quadric surface and K3 surfaces}, DOI={10.1017/s1474748021000384}, journal={Journal of the Institute of Mathematics of Jussieu}, author={Ascher, KENNETH and DeVleming, KRISTIN and Liu, YUCHEN}, year={2021}, pages={1–41}}

@article {Gallardo_2020,
    AUTHOR = {Gallardo, Patricio and Martinez-Garcia, Jesus and Spotti,
              Cristiano},
     TITLE = {Applications of the moduli continuity method to log {K}-stable
              pairs},
   JOURNAL = {J. Lond. Math. Soc. (2)},
  FJOURNAL = {Journal of the London Mathematical Society. Second Series},
    VOLUME = {103},
      YEAR = {2021},
    NUMBER = {2},
     PAGES = {729--759},
      ISSN = {0024-6107},
   MRCLASS = {32Q20 (14D22 14J10 14J45 14L24)},
  MRNUMBER = {4230917},
MRREVIEWER = {P. E. Newstead},
       DOI = {10.1112/jlms.12390},
       URL = {https://0-doi-org.serlib0.essex.ac.uk/10.1112/jlms.12390},
}

@misc{me_junyan_jesus,
      title={K-moduli of log del Pezzo pairs and Variation of GIT}, 
      author={Jesus Martinez-Garcia and Theodoros Papazachariou and Junyan Zhao},
      year={2023},
      }

@article {jiang,
    AUTHOR = {Jiang, Chen},
     TITLE = {Boundedness of {$\mathbb Q$}-{F}ano varieties with degrees and
              alpha-invariants bounded from below},
   JOURNAL = {Ann. Sci. \'{E}c. Norm. Sup\'{e}r. (4)},
  FJOURNAL = {Annales Scientifiques de l'\'{E}cole Normale Sup\'{e}rieure. Quatri\`eme
              S\'{e}rie},
    VOLUME = {53},
      YEAR = {2020},
    NUMBER = {5},
     PAGES = {1235--1248},
      ISSN = {0012-9593},
   MRCLASS = {14J45 (14L24)},
  MRNUMBER = {4174851},
       DOI = {10.24033/asens.244},
       URL = {https://0-doi-org.serlib0.essex.ac.uk/10.24033/asens.244},
}

@article {codogni,
    AUTHOR = {Codogni, Giulio and Patakfalvi, Zsolt},
     TITLE = {Positivity of the {CM} line bundle for families of {K}-stable
              klt {F}ano varieties},
   JOURNAL = {Invent. Math.},
  FJOURNAL = {Inventiones Mathematicae},
    VOLUME = {223},
      YEAR = {2021},
    NUMBER = {3},
     PAGES = {811--894},
      ISSN = {0020-9910},
   MRCLASS = {14J45 (14C20)},
  MRNUMBER = {4213768},
       DOI = {10.1007/s00222-020-00999-y},
       URL = {https://0-doi-org.serlib0.essex.ac.uk/10.1007/s00222-020-00999-y},
}

@article{blum_halpern-leistner_liu_xu_2021, title={On properness of K-moduli spaces and optimal degenerations of Fano varieties}, volume={27}, DOI={10.1007/s00029-021-00694-7}, number={4}, journal={Selecta Mathematica}, author={Blum, Harold and Halpern-Leistner, Daniel and Liu, Yuchen and Xu, Chenyang}, year={2021}}

@article {xu_valuations,
    AUTHOR = {Xu, Chenyang},
     TITLE = {A minimizing valuation is quasi-monomial},
   JOURNAL = {Ann. of Math. (2)},
  FJOURNAL = {Annals of Mathematics. Second Series},
    VOLUME = {191},
      YEAR = {2020},
    NUMBER = {3},
     PAGES = {1003--1030},
      ISSN = {0003-486X},
   MRCLASS = {14E30 (14J17 14J45)},
  MRNUMBER = {4088355},
MRREVIEWER = {Yuchen Liu},
       DOI = {10.4007/annals.2020.191.3.6},
       URL = {https://0-doi-org.serlib0.essex.ac.uk/10.4007/annals.2020.191.3.6},
}

@article{blum2021openness, 
title={Openness of K-semistability for Fano varieties}, 
volume={-1}, 
DOI={10.1215/00127094-2022-0054}, number={-1}, 
journal={Duke Mathematical Journal}, 
author={Blum, Harold and Liu, Yuchen and Xu, Chenyang}, 
year={2022}
}

@article {xu-zhuang,
    AUTHOR = {Xu, Chenyang and Zhuang, Ziquan},
     TITLE = {On positivity of the {CM} line bundle on {K}-moduli spaces},
   JOURNAL = {Ann. of Math. (2)},
  FJOURNAL = {Annals of Mathematics. Second Series},
    VOLUME = {192},
      YEAR = {2020},
    NUMBER = {3},
     PAGES = {1005--1068},
      ISSN = {0003-486X},
   MRCLASS = {14J45 (14D20 14E30)},
  MRNUMBER = {4172625},
MRREVIEWER = {Kenta Hashizume},
       DOI = {10.4007/annals.2020.192.3.7},
       URL = {https://0-doi-org.serlib0.essex.ac.uk/10.4007/annals.2020.192.3.7},
}

@article {xu2020uniqueness,
    AUTHOR = {Xu, Chenyang and Zhuang, Ziquan},
     TITLE = {Uniqueness of the minimizer of the normalized volume function},
   JOURNAL = {Camb. J. Math.},
  FJOURNAL = {Cambridge Journal of Mathematics},
    VOLUME = {9},
      YEAR = {2021},
    NUMBER = {1},
     PAGES = {149--176},
      ISSN = {2168-0930},
   MRCLASS = {14B05 (13A18 14E30)},
  MRNUMBER = {4325260},
       DOI = {10.4310/CJM.2021.v9.n1.a2},
       URL = {https://0-doi-org.serlib0.essex.ac.uk/10.4310/CJM.2021.v9.n1.a2},
}

@article{Blum_Xu_2019, title={Uniqueness of K-polystable degenerations of Fano varieties}, volume={190}, DOI={10.4007/annals.2019.190.2.4}, number={2}, journal={Annals of Mathematics}, author={Blum, Harold and Xu, Chenyang}, year={2019}}

@article {alper_reductivity,
    AUTHOR = {Alper, Jarod and Blum, Harold and Halpern-Leistner, Daniel and
              Xu, Chenyang},
     TITLE = {Reductivity of the automorphism group of {$K$}-polystable
              {F}ano varieties},
   JOURNAL = {Invent. Math.},
  FJOURNAL = {Inventiones Mathematicae},
    VOLUME = {222},
      YEAR = {2020},
    NUMBER = {3},
     PAGES = {995--1032},
      ISSN = {0020-9910},
   MRCLASS = {14J45 (14D23 14E30 14J10)},
  MRNUMBER = {4169054},
       DOI = {10.1007/s00222-020-00987-2},
       URL = {https://0-doi-org.serlib0.essex.ac.uk/10.1007/s00222-020-00987-2},
}

@article{liu2021finite, title={Finite Generation for valuations computing stability thresholds and applications to K-stability}, volume={196}, DOI={10.4007/annals.2022.196.2.2}, number={2}, journal={Annals of Mathematics}, author={Liu, Yuchen and Xu, Chenyang and Zhuang, Ziquan}, year={2022}}

@article {xu2020kstability,
    AUTHOR = {Xu, Chenyang},
     TITLE = {K-stability of {F}ano varieties: an algebro-geometric
              approach},
   JOURNAL = {EMS Surv. Math. Sci.},
  FJOURNAL = {EMS Surveys in Mathematical Sciences},
    VOLUME = {8},
      YEAR = {2021},
    NUMBER = {1-2},
     PAGES = {265--354},
      ISSN = {2308-2151},
   MRCLASS = {14D20 (14E30 14J10 14J45)},
  MRNUMBER = {4307210},
       DOI = {10.4171/emss/51},
       URL = {https://0-doi-org.serlib0.essex.ac.uk/10.4171/emss/51},
}

@article{Alper_good, title={Good moduli spaces for Artin stacks}, volume={63}, DOI={10.5802/aif.2833}, number={6}, journal={Annales de l’institut Fourier}, author={Alper, Jarod}, year={2013}, pages={2349–2402}}

@misc{stacks-project,
  author       = {The {Stacks project authors}},
  title        = {The Stacks project},
  howpublished = {\url{https://stacks.math.columbia.edu}},
  year         = {2024},
}

@article{Alper_Hall_Rydh_2020, title={A Luna {\'e}tale slice theorem for algebraic stacks}, volume={191}, DOI={10.4007/annals.2020.191.3.1}, number={3}, journal={Annals of Mathematics}, author={Alper, Jarod and Hall, Jack and Rydh, David}, year={2020}}

@misc{zhao2023moduli,
      title={Moduli of Genus Six Curves and K-stability}, 
      author={Junyan Zhao},
      year={2023},
      eprint={2212.06992},
      archivePrefix={arXiv},
      primaryClass={math.AG}
}

@misc{kollar2009hulls,
      title={Hulls and Husks}, 
      author={J{\'a}nos Koll{\'a}r},
      year={2009},
      eprint={0805.0576},
      archivePrefix={arXiv},
      primaryClass={math.AG}
}

@article{Tian_Yau_1987, title={K{\"a}hler-Einstein metrics on complex surfaces with $C_1>0$}, volume={112}, DOI={10.1007/bf01217685}, number={1}, journal={Communications in Mathematical Physics}, author={Tian, Gang and Yau, Shing-Tung}, year={1987}, pages={175–203}}

@article{Tian_1990, title={On Calabi’s conjecture for complex surfaces with positive first Chern class}, volume={101}, DOI={10.1007/bf01231499}, number={1}, journal={Inventiones Mathematicae}, author={Tian, G.}, year={1990}, pages={101–172}}

@article{Li_2018_products, title={Deformation of the product of complex Fano manifolds}, volume={356}, DOI={10.1016/j.crma.2018.04.007}, number={5}, journal={Comptes Rendus. Mathématique}, author={Li, Qifeng}, year={2018}, pages={538–541}}

@article{Fujita_2021_hyp, title={K-stability of log Fano hyperplane arrangements}, volume={30}, DOI={10.1090/jag/783}, number={4}, journal={Journal of Algebraic Geometry}, author={Fujita, Kento}, year={2021},  pages={603–630}}

@book{Hartshorne_2010, place={New York}, title={Algebraic geometry}, publisher={Springer}, author={Hartshorne, Robin}, year={2010}}

@book {grothendieck,
    AUTHOR = {Grothendieck, Alexander},
     TITLE = {Rev\^{e}tements \'{e}tales et groupe fondamental. {F}asc. {I}:
              {E}xpos\'{e}s 1 \`a 5},
      NOTE = {Troisi\`eme \'{e}dition, corrig\'{e}e,
              S\'{e}minaire de G\'{e}om\'{e}trie Alg\'{e}brique, 1960/61},
 PUBLISHER = {Institut des Hautes \'{E}tudes Scientifiques, Paris},
      YEAR = {1963},
     PAGES = {iv+143 pp. (not consecutively paged) (loose errata)},
   MRCLASS = {14.55},
  MRNUMBER = {217087},
}

@article {uniqueness_K-ps,
    AUTHOR = {Blum, Harold and Xu, Chenyang},
     TITLE = {Uniqueness of {K}-polystable degenerations of {F}ano
              varieties},
   JOURNAL = {Ann. of Math. (2)},
  FJOURNAL = {Annals of Mathematics. Second Series},
    VOLUME = {190},
      YEAR = {2019},
    NUMBER = {2},
     PAGES = {609--656},
      ISSN = {0003-486X,1939-8980},
   MRCLASS = {14J45 (14D20 14E30)},
  MRNUMBER = {3997130},
MRREVIEWER = {James\ McKernan},
       DOI = {10.4007/annals.2019.190.2.4},
       URL = {https://doi.org/10.4007/annals.2019.190.2.4},
}

@article {K-moduli_properness,
    AUTHOR = {Blum, Harold and Halpern-Leistner, Daniel and Liu, Yuchen and
              Xu, Chenyang},
     TITLE = {On properness of {K}-moduli spaces and optimal degenerations
              of {F}ano varieties},
   JOURNAL = {Selecta Math. (N.S.)},
  FJOURNAL = {Selecta Mathematica. New Series},
    VOLUME = {27},
      YEAR = {2021},
    NUMBER = {4},
     PAGES = {Paper No. 73, 39},
      ISSN = {1022-1824,1420-9020},
   MRCLASS = {14J45 (14D23)},
  MRNUMBER = {4292783},
MRREVIEWER = {Kenta\ Hashizume},
       DOI = {10.1007/s00029-021-00694-7},
       URL = {https://doi.org/10.1007/s00029-021-00694-7},
}

@article {kaloghiros-petracci,
    AUTHOR = {Kaloghiros, Anne-Sophie and Petracci, Andrea},
     TITLE = {On toric geometry and {K}-stability of {F}ano varieties},
   JOURNAL = {Trans. Amer. Math. Soc. Ser. B},
  FJOURNAL = {Transactions of the American Mathematical Society. Series B},
    VOLUME = {8},
      YEAR = {2021},
     PAGES = {548--577},
      ISSN = {2330-0000},
   MRCLASS = {14J45 (14B07 14D23 14M25)},
  MRNUMBER = {4287508},
MRREVIEWER = {Fei\ Hu},
       DOI = {10.1090/btran/82},
       URL = {https://doi.org/10.1090/btran/82},
}

\end{document}